\documentclass[11pt]{article}
 \usepackage{amsmath, amssymb, amsfonts, amsthm, enumerate, tikz, a4wide}
\date{}
\usepackage[margin= 21mm,bottom=21mm, top= 20mm]{geometry}

\newcommand\blfootnote[1]{%
  \begingroup
  \renewcommand\thefootnote{}\footnote{#1}%
  \addtocounter{footnote}{-1}%
  \endgroup
}

\newcommand{\Exp}{\mathbb{E}}
\newcommand{\Prob}{\mathbb{P}}

\newcommand{\Comment}[1]{\textbf{[#1]}}
\renewcommand{\Comment}[1]{}

\newcommand{\Hide}[1]{}

\newtheorem{theorem}{\bf Theorem}[section]

\newtheorem{conjecture}[theorem]{\bf Conjecture}

\newtheorem{definition}[theorem]{\bf Definition}
\newtheorem{question}[theorem]{\bf Question}

\newtheorem{fact}[theorem]{\bf Fact}

\newtheorem{claim}[theorem]{\bf Claim}
\newtheorem{lemma}[theorem]{\bf Lemma}
\newtheorem{problem}[theorem]{\bf Problem}

\title{Hyperstability in the Erd\H{o}s-S\'os Conjecture}
\date{}

\author{
Alexey Pokrovskiy}

\begin{document} 
\maketitle
\begin{abstract}
A rough structure theorem is proved for graphs $G$ containing no copy of a bounded degree tree $T$: from any such $G$, one can delete $o(|G||T|)$ edges in order to get a subgraph each of whose connected components have a vertex cover of order $3|T|$.

This theorem has the ability to turn questions about sparse $T$-free graphs (about which relatively little is known), into questions about dense $T$-free graphs (for which we have powerful techniques like regularity). There are various applications, the most notable being a proof of the Erd\H{o}s-S\'os Conjecture for large, bounded degree trees.
\end{abstract}

\section{Introduction}\blfootnote{Department of Mathematics, University College London, Gower Street, London WC1E 6BT, UK. \\ \emph{Email}: \texttt{dralexeypokrovskiy@gmail.com.}}
Extremal combinatorics is often summarized ``estimate how large a certain parameter can be ranging over a certain family of graphs''. A classic starting point for this is estimating ``extremal numbers'' or ``Tur\'an numbers'' of graphs. These are defined as the maximum   number of edges an $n$-vertex graph $G$ can have whilst containing no copies of $H$ i.e. $ex(n,H):=\max(e(G): G \text{ has no copies of $H$})$. 
Erd\H{o}s and Stone showed that $ex(n,H)$ is largely controlled by the chromatic number of $H$ by proving $ex(n, H)=\frac12(1-\frac1{\chi(H)-1})n^2 +o(n^2)$. For non-bipartite graphs (i.e. ones with $\chi(H)\geq 3$), this  solves the question of determining their Tur\'an numbers --- at least up the the lower order term $o(n^2)$. But for bipartite $H$, the Erd\H{o}s-Stone Theorem only gives ``$ex(n, H)=o(n^2)$'', leaving lots to discover. For bipartite $H$ containing at least one cycle, it is known that there exist $a(H), b(H)\in (1,2)$ so that $\Omega(n^{a(H)})\leq ex(n, H)\leq O(n^{b(H)})$ ($b(H)$ exists by the Kovari-S\'os-Tur\'an Theorem~\cite{kHovari1954problem}, while $a(H)$ exists by a standard probabilistic deletion argument). It is generally believed that for such $H$, there exists a rational exponent $\alpha\in (1,2)$ such that $ex(n,H)=\Theta(n^{\alpha})$ (this is called ``the Rational Exponents Conjecture''). 

The above was about bipartite graphs \emph{containing a cycle} --- but what about ones which have no cycles i.e. trees and forests? Here things change because determining the exponent of the Tur\'an number suddenly becomes easy: when $n$ divides $d$, one can show that for any  $d$-edge tree, we have $(d-1)n/2\leq ex(n,T)\leq d n$. For the lower bound, consider $n/d$ vertex-disjoint complete subgraphs of order $d$ --- this has $(n/d)\binom {d}2=(d-1)n/2$ but no copies of $T$ (since the connected components don't have enough vertices to house $T$). For the upper bound, it is standard that any graph with $dn$ edges has a subgraph $H$ with $\delta(H)\geq d$ (see Fact~\ref{fact:subgraph_min_degree_e/2}). One can then  greedily embed any  $d$-edge tree into $H$ one vertex at a time (see Fact~\ref{fact:greedy_tree_embedding}). 

The above paragraph solves one mystery about $ex(n,T)$ (determining the exponent of  $ex(n,T)=\Theta(dn^1)$), but immediately replaces it with other ones: does the limit $\lim_{n\to \infty}ex(n,T)/n$ exist and, if so, what is it? Or more broadly, what does  $ex(n,T)$  actually equal? In 1964, Erd\H{o}s and S\'os made a conjecture about what the answer should be:
\begin{conjecture}[Erd\H{o}s, S\'os, see \cite{erdos1964extremal}]\label{conj:Erdos_Sos}
For all $d$-edge trees $T$, $ex(n,T)\leq (d-1)n/2$.
\end{conjecture}
In other words Erd\H{o}s and S\'os predicted that the lower bound construction described above ($n/d$ vertex-disjoint complete graphs of order $d$) should be tight. One difficulty with this conjecture is that this extremal construction is non-unique. If $T$ is a star, then any $(d-1)$-regular graph contains no $T$ and has $(d-1)n/2$ edges. More broadly, if $T$ is balanced (i.e. both parts of its bipartition have order $(d+1)/2$) then let $G$ be a graph containing $(d-1)/2$ vertices which are joined to everything, and no other edges. Then $e(G)=n(d-1)/2-(d-1)(d-5)/8$, which, for large $n$, shows that $G$ is near-extremal for the Erd\H{o}s-S\'os Conjecture. There are of course other near-extremal graphs one can build by taking vertex-disjoint unions of the various constructions described above.

Much is known about the conjecture. In 1959, predating the conjecture, Erd\H{o}s and Gallai proved it for paths. For stars, the conjecture is an easy exercise. There are more early results about specific classes of trees~\cite{borowiecki1993erdos}, and about $n=d+1$~\cite{zhou1984note}. Then, in the early 90s, Ajtai, Koml\'os, Simonovits, and Szemer\'edi announced a full proof of the conjecture for sufficiently large $d$.
\begin{theorem}[Ajtai, Koml\'os, Simonovits, and Szemer\'edi]\label{thm:Ajtai_Komlos_Simonovits_Szemeredi}
For large enough $d$,  all $d$-edge trees $T$ have $ex(n,T)\leq(d-1)n/2$.
\end{theorem}
A paper containing this proof has not appeared yet. However some details about roughly how it goes are available. Some slides describing the proof can be found here~\cite{Komlos_Slides}. Additionally the series of papers  \cite{hladky2017approximate1, hladky2017approximate2, hladky2017approximate3, hladky2017approximate4}, though not directly about   Erd\H{o}s-S\'os, say that they use some ideas from the proof of Theorem~\ref{thm:Ajtai_Komlos_Simonovits_Szemeredi}. From these one can see that in order to prove Theorem~\ref{thm:Ajtai_Komlos_Simonovits_Szemeredi}, Ajtai, Koml\'os, Simonovits, and Szemer\'edi developed the regularity method as well as accompanying techniques like stability/extremal-case analysis. Their ideas have been hugely influential on combinatorics in the subsequent years, and inspire the current paper as well. 

Despite the announcement of Theorem~\ref{thm:Ajtai_Komlos_Simonovits_Szemeredi}, in the intervening decades there has been interest in replicating its proof. This is likely due to the importance of the techniques introduced by Ajtai, Koml\'os, Simonovits, and Szemer\'edi, as well as the fact that Erd\H{o}s-S\'os is not just an isolated conjecture --- there are many interesting variants of it (see the concluding remarks, or the survey~\cite{stein2020tree}).
The conjecture has been proven for specific families of trees~\cite{ fan2018erd, fan2007erdos,mclennan2005erdHos,tiner2015erdos}, for specific host graphs that forbid various subgraphs~\cite{balasubramanian2007constructing, brandt1996erdHos, dobson2002constructing, eaton2013erdHos, sacle1997erdHos,yin2004erdHos}, and when $n$ is very close to $d$~\cite{besomi2021erdHos, gorlich2016erdHos,wozniak1996erdos, yuan2014erdos}. 
The most general published results about the conjecture focus on the dense case i.e. when $e(G)\geq \Omega (|G|^2)$ or equivalently when $|G|\leq O(d)$. In this case the conjecture has been proved approximately for trees with $\Delta(T)\leq o(|T|)$ via two independent proofs --- one by Rozho\v{n}~\cite{rozhon2019local} and one by Besomi,  Pavez-Sign{\'e}, Stein   \cite{besomi2021erdHos}.  Besomi,  Pavez-Sign{\'e}, Stein  combined their ideas with an extremal case analysis in order to prove the conjecture exactly in the case when $G$ is dense and $T$ has bounded degree~\cite{besomi2021erdHos}. Recently, a proof of the approximate version of the conjecture has been announced for dense graphs without degree restrictions by Davoodi, Piguet, {\v{R}}ada,  and Sanhueza-Matamala~\cite{davoodi2023beyond}.  

From the above, we see that some of the ideas of Ajtai, Koml\'os, Simonovits, and Szemer\'edi have now been rediscovered. However one notable gap exists --- the most powerful existing papers about the Erd\H{o}s-S\'os Conjecture are about \emph{dense} graphs. Not much is known about general sparse graphs (with one very notable exception in the series of papers \cite{hladky2017approximate1, hladky2017approximate2, hladky2017approximate3, hladky2017approximate4} which prove the closely related Loebl-Koml\'os-S\'os Conjecture asymptotically, even for sparse graphs). 

The purpose of the current paper is purely to investigate the sparse case, and to provide a tool which can help with understanding the Erd\H{o}s-S\'os Conjecture as well as   its variants. The tool is the following rough structure theorem for graphs without bounded degree trees. Recall that a ``cover'' of a graph $G$ is a set  of vertices $C\subseteq V(G)$ such that all edges $xy\in E(G)$ have $\{x,y\}\cap C\neq \emptyset$.
\begin{theorem}\label{main_theorem}
Let $\Delta\in \mathbb{N}, \varepsilon>0$, and let $d$ be sufficiently large. 
Let $T$ be a tree with $d$ edges and $\Delta(T)\leq \Delta$. For any  graph $G$ with $e(G)\geq \varepsilon d|G|$ having no copies of $T$, it is possible to delete $\varepsilon e(G)$ edges to get a graph $H$ each of whose connected components has a cover of order $\leq 3d$.
\end{theorem}
It is not immediately clear how one should think of the above theorem.   
Perhaps the closest point of contact to Theorem~\ref{main_theorem} is that of \emph{stability theorems} --- theorems that say things like ``if a graph $G$ has $e(G)\geq ex(n,H)- x$, then $G$ must have some constrained structure (which depends on $H$ and $x$)''. The constrained structure usually means ``$G$ looks close to the extremal graph with $ex(n,H)$ edges that contains no copy of $H$''. Indeed Theorem~\ref{main_theorem} has this flavour --- the extremal and near-extremal graphs described earlier with $e(G)\approx (d-1)n/2$ all have   the feature that their connected components have an $O(d)$-sized cover, and Theorem~\ref{main_theorem} exactly says that $G$ looks approximately like this. The unexpected thing about Theorem~\ref{main_theorem} is that it applies even when $e(G)$ is \emph{far below  $ex(n,T)$}. Even when $G$ has $1\%$ of the edges of an extremal $T$-free graph,  it turns out that the only way to avoid copies of $T$ is basically the same as if $G$ had $ex(n,T)$ edges. To the author's knowledge it hasn't been previously observed that Tur\'an numbers can sometimes exhibit this much stability --- and indeed it is likely that nothing similar holds when $T$ is not a forest. We termed this behaviour ``hyperstability'' in analogy to a similar term which is used for describing dynamical systems~\cite{brzdkek2013hyperstability}.

One  weakness of Theorem~\ref{main_theorem} is that it is only for bounded degree trees. This is not entirely a shortcoming of the proof --- the theorem is simply false without some bounded degree assumption on $T$. Indeed if $T$ has a vertex of degree $\varepsilon d$, then any $(\varepsilon d-1)$-regular graph $G$ contains no copy of $T$ --- but it doesn't necessarily have the structure of Theorem~\ref{main_theorem}. Of course there is a lot of room between this example (which has $\Delta(T)\leq \varepsilon d$), and the bounded degree restriction  ``$\Delta(T)\leq \Delta$'' in Theorem~\ref{main_theorem}. Understanding what happens in between is an interesting open problem.

We remark that the upper bound on size of the cover ``$3d$'', while close to being tight (the correct bound is $(1+o(1))d$ --- see concluding remarks), is not actually an important feature of the theorem. Indeed if this was replaced by any fixed multiple of $d$ like ``$1000d$'' then the theorem would be equally useful in the applications we have. In fact even if this number was allowed to increase with $\varepsilon$ (e.g. if the theorem upper bounded the cover sizes by $\varepsilon^{-\varepsilon^{-\varepsilon}}d$), then most applications would still go through.
We also remark that there is an equivalent statement of the theorem where the assumption $e(G)\geq \varepsilon d|G|$ is omitted, and instead the number of deleted edges is  $\varepsilon d |G|$ (see Theorem~\ref{main_theorem_proof}).

If the numbers aren't important, then what actually is the point of Theorem~\ref{main_theorem}? The power of the theorem lies in its ability to transform questions about sparse $T$-free graphs (about which very little is known) into questions about dense $T$-free graph (about which there is a large body of research, and where we have powerful tools like regularity). This happens because the input graph $G$ in Theorem~\ref{main_theorem} can have as few as $\varepsilon d|G|$ edges, which is significantly sparser than what the Erd\H{o}s-S\'os Conjecture (or any of its variants) ask for. On the other hand the components of the output graph $H$, while not quite dense themselves, turn out to be amenable to a broad range of techniques that work for dense graphs. For example, the simplest trick to get a dense graph from Theorem~\ref{main_theorem},  is to first observe that by averaging, one of the components of $H$ must have $e(C)/|C|\geq e(H)/|H|\geq (1-\varepsilon)e(G)/|G|\geq \varepsilon d/2$. Next, letting $D$ be the  $3d$-sized cover of $C$,  randomly select a set $R$ of $\max(|C\setminus D|,1000d)$ vertices outside $D$. Now the induced subgraph $C[R\cup D]$ has order $\leq 1003d$, and $e(C[R\cup D])\geq e(C)\frac{|R|}{|C|}\geq \varepsilon d|R|/2\geq \varepsilon |R\cup D|^2/2000$ in expectation i.e. $C[R\cup D]$ is dense. Even this very simple trick is already powerful enough to prove the Erd\H{o}s-S\'os Conjecture approximately for bounded degree trees (when combined with existing techniques about dense graphs. See Section~\ref{Section_approximate_Erdos_Sos} for details).

In the subsections which follow, we present some initial applications of Theorem~\ref{main_theorem} in order to illustrate the sorts of ways it can be used. Then, from Section 2 onwards we present the proof of Theorem~\ref{main_theorem}. Some further, more intricate applications will appear in joint work with Versteegen and Williams~\cite{vw}.

\subsubsection*{Formalization in Lean} 
The following weaker version of our main theorem has been formally verified using the Lean proof assistant, and can be found at \cite{github}. 
\begin{theorem}\label{theorem_Lean}
Let $\varepsilon>0$, and let $C$ be sufficiently large. 
Let $P$ be a path with $d\ge C$ edges. For any  graph $G$ with $e(G)\geq \varepsilon d|G|$ having no copies of $P$, it is possible to delete $\varepsilon e(G)$ edges to get a graph $H$ each of whose connected components has a cover of order $\leq Cd$.
\end{theorem}
This is weaker than Theorem~\ref{main_theorem} in two ways. Firstly, the size of the covers is $Cd$ rather than $3d$. As mentioned before, this does not make much difference for the existing applications of Theorem~\ref{main_theorem} --- for example the approximate proof of the Erd\H{o}s-S\'os Conjecture in Section~\ref{Section_approximate_Erdos_Sos} would go through without any real changes using the bounds given by Theorem~\ref{theorem_Lean}. 
The second difference is that Theorem~\ref{theorem_Lean} is only about paths rather than general bounded degree trees (which obviously cuts down any applications to also only be about paths). Both of these weakenings happened to avoid having to formalize certain known theorems --- for example by working with just paths, one can avoid needing the Friedman-Pippenger Theorem as well as   tools for embedding trees in dense graphs (Lemma~\ref{lem:main_tree_embedding_lemma}). Modulo these standard results, the formalized proof of Theorem~\ref{theorem_Lean} is the same as that given for Theorem~\ref{main_theorem}. In the future, the author hopes to complete the formalization to a proof of the full Theorem~\ref{main_theorem}.

\subsection{Application: approximate Erd\H{o}s-S\'os}\label{Section_approximate_Erdos_Sos}
Several  proofs of  approximate versions of the Erd\H{o}s-S\'os Conjecture for dense graphs exist. They vary in the maximum degree of the trees they apply to, but all of them imply the following.
\begin{theorem}[Rozho\v{n}\cite{rozhon2019local} Theorem 1.2. Besomi,  Pavez-Sign{\'e}, Stein \cite{stein2020tree} Theorem 1.9,  \cite{besomi2021erdHos} Theorem 1.3]\label{thm:Erdos_Sos_dense_approximate}
For $\Delta\in \mathbb{N}, \varepsilon>0$, there exists $d_0$ such that the following holds for $d\ge d_0$ and $n\leq \varepsilon^{-1}d$. 
Let $G$ be an $n$-vertex graph with $e(G)> (1+\varepsilon)dn/2$. Then $G$ contains a copy of every $d$-edge tree $T$ having $\Delta(T)\leq \Delta$. 
\end{theorem}
This theorem can easily be black-boxed together with our hyperstability theorem to extend it to the sparse case.
\begin{theorem}\label{thm:Erdos-Sos-approximate}
For $\Delta\in \mathbb{N}, \varepsilon>0$, there exists $d_1$ such that the following holds for $d\ge d_1$.  
Let $G$ be an $n$-vertex graph with $e(G)> (1+\varepsilon)dn/2$. Then $G$ contains a copy of every $d$-edge tree $T$ having $\Delta(T)\leq \Delta$. 
\end{theorem}
\begin{proof}
Set $C:=80\varepsilon^{-1}$ and $\varepsilon_0= (C+3)^{-1}$, noting $\varepsilon_0\le C^{-1}\le \varepsilon/4$. Apply Theorem~\ref{thm:Erdos-Sos-approximate}  with $\varepsilon_0= (C+3)^{-1}, \Delta=\Delta$ to get a $d_0$. Let $d_1$ be the maximum of $d_0$ and how large $d$ needs to be in Theorem~\ref{main_theorem}, when applied with $\varepsilon_1=\varepsilon/3, \Delta=\Delta$. Now let $d\ge d_1$, and let $G$ be an $n$-vertex graph with $e(G)> (1+\varepsilon)dn/2$.   

Note that $e(G)> (1+\varepsilon)dn/2\ge (\varepsilon/3) dn$, so we can apply Theorem~\ref{main_theorem} with $\varepsilon_1=\varepsilon/3$ in order to find a subgraph $G'$ having $e(G')\geq (1-\varepsilon/3)e(G)\geq (1-\varepsilon/3)(1+\varepsilon)dn/2\geq  (1+\varepsilon/2)dn/2$ whose connected components each have a cover of order $\leq 3d$. 
By averaging (see Fact~\ref{fact:connectected_component_with_same_density_as_G}), some connected component $K$ of $G'$ has $e(K)\geq (1+\varepsilon/2)d|K|/2\ge (1+\varepsilon_0)d|K|/2$.
Let $D$ be the cover of order $\leq 3d$ of $K$. 
We can assume $|K|> (C+3)d=\varepsilon_0^{-1}d$, since otherwise  we are done by Theorem~\ref{thm:Erdos_Sos_dense_approximate}.
Let $R$ be a random set of $Cd$ vertices from $K\setminus D$, noting that $|R\cup D|\le Cd+3d=(C+3)d=\varepsilon_0^{-1}d$. We have $\Exp(e(R\cup D))\geq e(K)\frac{Cd}{|K|}\geq \frac{Cd}{|K|}(1+\varepsilon/2)d|K|/2=(1+\varepsilon/2)d(Cd)/2\geq \frac{Cd}{(C+3)d}(1+\varepsilon/2)d|R\cup D|/2=\frac{1}{1+3C^{-1}}(1+\varepsilon/2)d|R\cup D|/2\geq (1+\varepsilon/2-20C^{-1})d|R\cup D|/2=(1+\varepsilon/4)d|R\cup D|/2\ge (1+\varepsilon_0)d|R\cup D|/2$. Now Theorem~\ref{thm:Erdos_Sos_dense_approximate} applies to $G[R\cup D]$  (with $\varepsilon_0, \Delta$) to give a copy of $T$.
 \end{proof}

\subsection{Application: approximate Loebl-Koml\'os-S\'os}
The Erd\H{o}s-S\'os Conjecture is about the \emph{average degree} of the host graph $G$. Loebl-Koml\'os-S\'os made a variant of the conjecture replacing  average degree by the \emph{median degree}
\begin{conjecture}[Loebl-Koml\'os-S\'os, see \cite{erdHos1995discrepancy}]\label{conj:Loebl-Komlos-Sos}
Let $G$ be an $n$-vertex graph with $\geq n/2$ vertices of degree $\geq d$. Then $G$ contains a copy of every $d$-edge tree $T$. 
\end{conjecture}
This conjecture has a very rich history, see Section 3 of~\cite{stein2020tree} for a full list of references. The conjecture has been fully proved when the host graph $G$ is large and dense via two independent proofs by Cooley~\cite{cooley2009proof} and by Hladky, Piguet~\cite{hladky2016loebl}. Also Loebl-Koml\'os-S\'os is notable in being one of the only conjectures in this area where an approximate version was previously proved for sparse graphs. This was done by Hladky,  Koml{\'o}s,  Piguet,  Simonovits,  Stein,  and Szemer{\'e}di in a tour-de-force work spanning four papers~\cite{hladky2017approximate1, hladky2017approximate2, hladky2017approximate3, hladky2017approximate4}. 
\begin{theorem}[Hladky,  Koml{\'o}s,  Piguet,  Simonovits,  Stein,  and Szemer{\'e}di, \cite{hladky2017approximate1, hladky2017approximate2, hladky2017approximate3, hladky2017approximate4}]\label{thm:Loebl-Komlos-Sos_approximate}
Let $\varepsilon\gg d^{-1}$. 
Let $G$ be an $n$-vertex graph with $\geq (1+\varepsilon)n/2$ vertices of degree $\geq(1+\varepsilon)d$. Then $G$ contains a copy of every $d$-edge tree $T$. 
\end{theorem}

Define a tree $T$ to be  $r$-skew if one of its parts  has order $\leq r|T|$. Motivated by a question of Simonovits, Klimo\v{s}ov\'a, Piguet, and Rohzo\v{n} made the following sharpening of Conjecture~\ref{conj:Loebl-Komlos-Sos} which takes into account the skew of $T$.

\begin{conjecture}[Klimo\v{s}ov\'a, Piguet, and Rohzo\v{n}, \cite{klimovsova2020version}]\label{conj:skew_Loebl-Komlos-Sos}
Let $G$ be an $n$-vertex graph with $\geq rn$ vertices of degree $\geq d$. Then $G$ contains a copy of every $r$-skew $d$-edge tree $T$. 
\end{conjecture}

The authors of this conjecture proved it approximately for dense graphs.
\begin{theorem}[Klimo\v{s}ov\'a, Piguet, and Rohzo\v{n}, \cite{klimovsova2020version}]\label{thm:skew_Loebl-Komlos-Sos_dense_approximate}
Let $r^{-1}\gg \varepsilon\gg d^{-1}$ and  $n\leq \varepsilon^{-1}d$. 
Let $G$ be an $n$-vertex graph with $\geq rn$ vertices of degree $\geq (1+\varepsilon)d$. Then $G$ contains a copy of every $r$-skew $d$-edge tree $T$. 
\end{theorem}
Black-boxing this theorem together with our hyperstability theorem, we can extend it to the sparse, bounded degree case. This gives a generalization of Theorem~\ref{thm:Loebl-Komlos-Sos_approximate} --- though unfortunately only for bounded degree trees.
\begin{theorem}
Let $1\geq \Delta^{-1}, r, 1-r\gg \varepsilon\gg d^{-1}$. 
Let $G$ be an $n$-vertex graph with $(1+\varepsilon)rn$ vertices of degree $\geq (1+\varepsilon)d$. Then $G$ contains a copy of every $r$-skew $d$-edge tree $T$ having $\Delta(T)\leq \Delta$.
\end{theorem}
\begin{proof} 
 Let $L=\{v\in G: d_G(v)\geq (1+\varepsilon)d\}$, noting $|L|\geq (1+\varepsilon)rn$. 
Apply Theorem~\ref{main_theorem_proof} (with $\varepsilon'=\varepsilon^4/8$) in order to find a subgraph $G'$ having $e(G')\geq e(G)-\varepsilon^4 dn/8$ whose connected components $C_1, \dots, C_t$ each have a cover of order $\leq 3d$. 
For each $i$, let $B(C_i)=\{v\in C_i: d_G(v)\geq (1+\varepsilon)d, d_{G'}(v)< (1+\varepsilon/2)d\}$. Note that $e(G\setminus G')\geq \frac12 \sum_{i=1}^t |B(C_i)|\varepsilon d/2$ gives $\sum_{i=1}^t|B(C_i)|\leq 4(\varepsilon^4 dn/8)/\varepsilon d\leq \varepsilon^2 n/2$. Let $L_i=\{v\in C_i: d_{G'}(v)\geq (1+\varepsilon/2)d\}$.
We have $\sum_{i=1}^t |L_i|\geq \sum_{i=1}^t (|L\cap C_i|-|B(C_i)|)\geq (1+\varepsilon)rn-\varepsilon^2 n/2\geq  (1+\varepsilon/2)rn=(1+\varepsilon/2)r\sum |C_i|$, which implies that for some $C_i$, we have $|L_i|\geq (1+\varepsilon/2)r|C_i|$. 

If $|C_i|\leq 100\varepsilon^{-2}d$, then we are done by Theorem~\ref{thm:skew_Loebl-Komlos-Sos_dense_approximate} (applied with $\varepsilon'=\varepsilon^2/100$), so suppose $|C_i|> 100\varepsilon^{-2}d$. Let $D_i$ be the cover we have of $C_i$ with $|D_i|\leq 3d$. We have $|L_i\setminus D_i|\geq (1+\varepsilon/2)r|C_i|-3d\geq (1+\varepsilon/2)r|C_i|- \varepsilon r|C_i|/4= (1+\varepsilon/4)r|C_i|\geq 50\varepsilon^{-2}rd$. Let $U_i\subseteq L_i\setminus D_i$ be a subset of order $|U_i|= 50\varepsilon^{-2}rd$ and consider the induced subgraph $H_i:=C_i[D_i\cup U_i]$. Note that since $D_i$ is a cover of $C_i$ we have $d_{H_i}(u)=d_{C_i}(u)\geq (1+\varepsilon/2)d$ for all $u\in U_i$. Also $|U_i|=50\varepsilon^{-2}rd\geq (1-\varepsilon)(50\varepsilon^{-2}rd+ 3d)\geq (1-\varepsilon)(|U_i|+|D_i|)= (1-\varepsilon)|H_i|\geq r|H_i|$. Thus Theorem~\ref{thm:skew_Loebl-Komlos-Sos_dense_approximate} applies (with $\varepsilon'=\varepsilon^2/103$) to $H_i$ in order to give a copy of $T$.
 \end{proof}

\subsection{Application: exact Erd\H{o}s-S\'os}
For dense host graphs and bounded degree trees, there is a published proof of the Erd\H{o}s-S\'os Conjecture.
\begin{theorem}[Besomi, Guido-Pavez-Sign{\'e}, Stein, \cite{besomi2021erdHos}]\label{thm:Erdos_Sos_dense}
Let $\Delta^{-1}, C^{-1}\gg d^{-1}$ and  $n\leq Cd$. 
Let $G$ be an $n$-vertex graph with $e(G)> (d-1)n/2$. Then $G$ contains a copy of every $d$-edge tree $T$ having $\Delta(T)\leq \Delta$. 
\end{theorem}
Using our hyperstability theorem, we can generalize this proof to sparse host graphs --- and so obtain a proof of the Erd\H{o}s-S\'os Conjecture for all large bounded degree trees. Unlike our approximate applications above, we cannot simply ``black box'' Theorem~\ref{thm:Erdos_Sos_dense} in order to deduce the sparse case. This is unsurprising --- the rough structure theorem gives no information about the leftover $\varepsilon dn$ edges of $G\setminus H$, and to prove any tight result one can't simply ignore these edges. Thus, in addition to Theorems~\ref{main_theorem} and~\ref{thm:Erdos_Sos_dense}, we also require extra ingredients that analyse near-extremal graphs in the Erd\H{o}s-Sos Conjecture. The content of these is essentially all present in~\cite{besomi2021erdHos}. However, they are not actually stated explicitly --- so we provide a companion paper~\cite{supplementary} where all the intermediate results we need are stated and proved fully.

The first thing we need is a stability result --- that a graph with $\approx (d-1)n/2$ edges without a copy of $T$ must have structure very close to the extremal graphs for the conjecture.
\begin{theorem}[\cite{supplementary}]\label{lem:stability_main}
Let $\Delta\gg  \alpha \gg \varepsilon\gg  d^{-1}$.
Let $T$ be a $d$-vertex tree with $\Delta(T)\leq \Delta$.
Let $G$ be a graph with $e(G)\geq (1-\varepsilon)dn/2$, and let $D$ be a cover   of order $|D|\leq 10d$. If $G$ has no copy of $T$, then there is a subgraph $H$ of $G$ 
satisfying one of:
\begin{enumerate}[(i)]
\item $\delta(H)\geq (1-\alpha)d$ and $|H|\leq (1+\alpha)d$.
\item  $H$ is bipartite with parts $X,Y$ such that $|X|\leq (1+\alpha)d/2$, $|Y|\geq 6d$, $\delta(X)\geq (1-\alpha)|Y|$, $\delta(Y)\geq (1-\alpha)d/2$.
\end{enumerate}
\end{theorem}

This is paired with two ``near-extremal analysis'' results which prove the conjecture for graphs having structures (i) or (ii).

\begin{theorem}[\cite{supplementary}]\label{lem:extremal_analysis_bipartite}
Let $\Delta^{-7}/1000\geq \varepsilon\ge d^{-1}$.
Let $T$ be a $d$-edge tree with $\Delta(T)\leq \Delta$.
Let $G$ be a connected graph with $\delta(G)\geq d/2$. Suppose that $G$ contains a bipartite subgraph $B$ with parts $X, Y$ such that $|X|\leq (1+\varepsilon)d/2$, $|Y|\geq 6d$, $\delta_B(X)\geq (1-\varepsilon)|Y|$, $\delta_B(Y)\geq (1-\varepsilon)d/2$. 
Then $G$ contains a copy of $T$.
\end{theorem}
\begin{theorem}[\cite{supplementary}]\label{lem:extremal_analysis_nonbipartite}
Let $\Delta^{-7}/1000\geq \varepsilon\ge 3d^{-1}$.
Let $G$ be a connected graph with $|G|\geq (1+256\Delta^2\sqrt{\varepsilon})d$, $\delta(G)\geq 128\Delta^2\sqrt{\varepsilon} d$ containing a subgraph $K$ with $|K|\leq (1+\varepsilon)d$, $\delta(K)\geq (1-\varepsilon)d$. Then $G$ contains a copy of every $d$-edge tree having $\Delta(T)\leq \Delta$.
\end{theorem}

We combine the above four results with Theorem~\ref{main_theorem} to get the Erd\H{o}s-S\'os Conjecture for large bounded degree trees.
\begin{theorem}\label{thm:Erdos-Sos-bounded-degree}
Let $\Delta^{-1}\gg d^{-1}$.
Let $G$ be an $n$-vertex graph with $e(G)> (d-1)n/2$. Then $G$ contains a copy of every $d$-edge tree $T$ having $\Delta(T)\leq \Delta$. 
\end{theorem}
\begin{proof}
Let $\Delta^{-1}\gg \alpha\gg\varepsilon\gg d^{-1}$. 
We can assume that $G$ has $\delta(G)\geq \lfloor(d-1)/2\rfloor+1\geq (d-1)/2-1/2+1= d/2$ (via Fact~\ref{fact:subgraph_min_degree_e/2}) and connected (via Fact~\ref{fact:connectected_component_with_same_density_as_G}).
We can assume $|G|\geq 10d$, as otherwise we're done by Theorem~\ref{thm:Erdos_Sos_dense} applied with $C=10$. 
Apply Theorem~\ref{main_theorem} with $\varepsilon'=\varepsilon/2$   in order to find a subgraph $G'$ having $e(G')\geq (1-\varepsilon/2)e(G)\geq (1-\varepsilon/2)(d-1)n/2= (1-\varepsilon/3)(1-d^{-1})dn/2\geq (1-\varepsilon)dn/2$ whose connected components each have a cover of order $\leq 3d$. 
By Fact~\ref{fact:connectected_component_with_same_density_as_G}, some connected component $C$ of $G'$ has $e(C)\geq (1-\varepsilon)d|C|/2$. Apply Lemma~\ref{lem:stability_main} to get a subgraph $H$ of $C$ satisfying (i) or (ii). 
Using that $G$ is connected of order $\geq 10d\geq (1+256\Delta^2\alpha)d$ and has $\delta(G)\geq d/2\geq 128\Delta^2\alpha d$, we find a copy of $T$ using either Lemma~\ref{lem:extremal_analysis_bipartite} or Lemma~\ref{lem:extremal_analysis_nonbipartite}. 
\end{proof}

\section{Proof overview}
Tree embeddings into dense graphs are well understood, and we make use of them. A key definition for this paper is that of cut-density, a concept which was first introduced by Conlon, Fox, and Sudakov~\cite{conlon2014cycle}.
\begin{definition}
A graph $G$ is a $q$-cut-dense if every partition $V(G)=A\cup B$ has $e_G(A,B)\geq q|A||B|$.
\end{definition}
We think of cut-density as a ``connectivity'' parameter. 
The motivation for this is that $q$-cut-density is easily shown to be equivalent (after some reparametrization) to saying that ``$G$ has linear minimum degree and the cluster graph (in the sense of Szemeredi's Regularity Lemma) is connected''.  
A nice thing about $q$-cut-density is that it only involves one parameter and makes no reference to regularity or pseudorandomness --- and so the definition can give a very clean perspective on some aspects of dense graphs. The following lemma is how we embed trees in cut-dense graphs:

\begin{lemma}\label{lem:main_tree_embedding_lemma}
Let  $p, q,\varepsilon,\Delta^{-1}\gg L^{-1}\gg d^{-1}$
 and let $T$ be a order $d$ tree with $\Delta(T)\leq \Delta$. Let $G$ be a graph which has one of the following structures:
\begin{enumerate}[(i)]
\item  $G$ is a $n$-vertex, $q$-cut-dense graph,  and contains an  $pn$-regular, order $\geq (2+\varepsilon)d$ subgraph $R$.
\item $E(G)=G_1\cup \dots \cup G_t$ with each $G_i$ is a $\geq 16d/qL$-vertex $q$-cut-dense graph. There is an auxiliary $L$-ary tree $S$ of depth $L$ tree with $V(S)=\{G_1, \dots, G_t\}$ and a set of distinct vertices $\{u_{G_iG_j}: G_iG_j\in E(S)\}\subseteq V(G)$ such that  $V(G_i)\cap V(G_j)=\{u_{G_iG_j}\}$ for all edges $G_iG_j\in E(S)$ and $V(G_i)\cap V(G_j)=\emptyset$ for non-edges $G_iG_j\not\in E(S)$.
\end{enumerate}
  Then $G$ has a copy of $T$. 
\end{lemma}
Part (i) of the lemma is quite standard --- it is essentially the same as  the fact that large bounded-degree trees exist in dense graphs whose cluster graphs contain a large connected matching (which is true by first splitting the tree $T$ into small subtrees, and then embedding those into the connected matching via the Blow-Up Lemma. See~\cite{stein2020tree} for a concrete argument like this). Part (ii) is a new statement, but   can be proved by a similar approach. We prove Lemma~\ref{lem:main_tree_embedding_lemma} in the appendix. We use a different strategy that doesn't rely on regularity or pseudorandomness --- instead using cut-density directly.  

Having talked about dense graphs, we now turn to understanding sparse graphs. What can be said about a graph $G$ which is ``nowhere dense'' in a sense that it has no order $m$ subgraphs $H$ that are $q$-cut-dense? It is easily shown that if $\delta(G)/m$ is large, then such a $G$ must be an expander in a sense that $|N(S)|\geq 10|S|$ for all small $S\subseteq V(G)$  (see Lemma~\ref{lem:expansion_no_cut_dense}). But bounded degree tree embeddings into such expanders are well understood via the classic result of Friedman and Pippenger.
\begin{theorem}[Friedman, Pippenger, \cite{friedman1987expanding}]\label{thm:Friedman-Pippenger}
Let $T$ be a tree with $d$ vertices and $\Delta(T)\leq \Delta$. Let $G$ be a graph satisfying $|N(S)|\geq 10\Delta |S|$ for all $S\subseteq V(G)$ with $|S|\leq 10d$. Then $G$ has a copy of $T$. 
 \end{theorem} 
 
Thus the difficulty of e.g. the Erd\H{o}s-S\'os Conjecture lies in understanding graphs $G$ that are ``mixed'' in a sense that they contain lots of small dense subgraphs (while $G$ itself is in no way dense). To this end, say a graph $G$ is ``locally dense'' if we can partition $E(G)=H_1\cup \dots \cup H_t$ where each $H_i$ is $q$-cut-dense of order $m=\Theta(e(G)/|G|)$ (this is a close relative of the locally dense graphs defined in  \cite{hladky2017approximate1}). Using what we've said so far, proving Theorem~\ref{main_theorem} reduces to proving it for locally dense graphs (see the first paragraph of the proof of Theorem~\ref{main_theorem_proof} for details). So, for the rest of this section suppose  that $G$ is locally dense.

The overall strategy for proving Theorem~\ref{main_theorem} is to obtain some very rough structural information about $G$ via something we call a ``clumping argument''. This is an infection process a bit like graph bootstrap percolation. One starts with some initial collection of ``clumps'' $H_1, \dots, H_t$ (these are just the subgraphs handed to us by the definition of ``locally dense graph''). Then for as long as possible repeat the following operation: if there are two clumps $H,H'$ in the collection with $|V(H)\cap V(H')|\geq \varepsilon d$, then replace them with a new clump $H\cup H'$. This process must eventually terminate (since the number of clumps decreases at each step), call the final clumps $K_1, \dots, K_s$. By virtue of the clumping process terminating, we know that these are nearly-disjoint i.e. $|V(K_i)\cap V(K_j)|\leq \varepsilon d$ always. Additional properties of the final clumps can also be proved by (inductively) tracking things throughout the process. The main parameter we need to track is connectivity/cut-density. We do this, in part, via the following.
\begin{itemize}
\item [$(\ast)$] Let $H$, $H'$ be $q$-cut-dense graphs with $|V(H)\cap V(H')|\geq \varepsilon d$. Then $H\cup H'$ is $\frac{\varepsilon d}{4|H\cup H'|}q$-cut-dense.
\end{itemize} 
See Lemma~\ref{lem:cut_dense_union} for a proof of this. To understand the rest of the proof, it is illustrative to first think about how it could work if there was no loss of cut-density in $(\ast)$ i.e. under the (completely unrealistic) assumption that we could replace ``$q\frac{\varepsilon d}{4|H\cup H'|}$'' by ``$q$''. If this were true, then we would inductively know that the final clumps  $K_1, \dots, K_s$ are $q$-cut-dense. Via Lemma~\ref{lem:main_tree_embedding_lemma} (i), we get that each $K_i$ has no order $\geq (2+\varepsilon)d$ regular subgraph, which translates into having an approximate vertex-cover of order $\leq (2+\varepsilon)d\le 3d$ (basically a maximal order regular subgraph of $K_i$ will give such a cover. See Fact~\ref{fact:BK_small_covers} and Lemma~\ref{lem:clump_large_intersection_with_B/D} for how this works).  

Let $Y$ be the set of vertices that are in more than one clump $K_i$. 
Consider an auxiliary ``incidence'' bipartite graph $B$ with parts $X:=\{K_1, \dots, K_s\}$ and $Y:=V(G)$ where we join $y$ to $K_i$ whenever $y\in K_i$. If $e(B)=0$, then the clumps $K_i$ are vertex-disjoint and we are done. If $e(B)\leq \alpha d |V(G)|$ for appropriate $\alpha$, then it turns out to be possible to delete $\varepsilon d|V(G)|$ edges of $G$ in order to make the clumps vertex-disjoint --- and so we are also done (the calculation for this part takes some work. See Claim~\ref{claim:clump_ordering} and the paragraphs which follow it). So we can suppose $e(B)> \alpha d |V(G)|$.

The property ``$|V(K_i)\cap V(K_j)|\leq \varepsilon d$ always'' translates to vertices $K_i,K_j\in X$ having small common neighbourhoods. We've established that $e(B)\geq \Omega(d|B|)$  with pairs of vertices in $X$ having small common neighbourhoods, or equivalently $B$ has no copies of $K_{2, \varepsilon d}$ with smaller side in $X$. It is well known that in such graphs it is easy to find $1$-subdivisions of smaller graphs. For example a classic theorem of K\"uhn-Osthus says that in every $K_{s,s}$-free graph of average degree $\geq f(s,k)$, there is an induced $1$-subdivision of some graph of average degree $\geq k$. We prove a variant of this where we only look for a $1$-subdivision of a $L$-ary tree $T_L$, but are more concerned with having a good dependence between $f$ and $s,L$ (see Lemma~\ref{lem:find_tree_1_subdivision}). The end result is that we are able to show that $B$ contains a $1$-subdivision of a $L$-ary height $L$ tree $T_L$ for $L$ large and with subdivided vertices in $Y$. Examining what structure this implies on the clumps $K_i\in V(T_L)\cap X$, we see that the union of these clumps satisfy something very similar to Lemma~\ref{lem:main_tree_embedding_lemma} (ii). After some small modifications, we can apply that lemma and complete the proof. 

We've now described the entire proof, under the unrealistic assumption that no cut-density is lost when joining clumps together via $(\ast)$. Getting around this assumption is the technical heart of the paper. The difficulty is that there is no way to bound how many times a clump might be involved in a joining operation --- and so it seems like any loss of cut-density in $(\ast)$   should break the proof. The fix to this is to track something other than cut-density throughout the joining operations. Throughout the clumping process, we track two subgraphs $M(K)$ and $C(K)$ for each clump $K$. Here $M(K)$ is roughly defined as ``a maximal order $d'$-regular subgraph  that $K$ contains'' (for suitable $d'$), while $C(K)$ is just a small $q(K)$-cut-dense subgraph that contains $M(K)$ for suitable $q(K)$.  
See Definition~\ref{def:clump} for the formal version of this.  With this in mind, the joining operation $(\ast)$ changes. Whenever we have two clumps with $|V(K)\cap V(K')|\geq \varepsilon d$, we are able to define a clump $K''=K\cup K'$ along with some $M(K''), C(K''), q(K'')$ satisfying one of the following:
\begin{enumerate}[(1)]
\item $|M(K'')|\geq \max (|M(K)|, |M(K')|)+\varepsilon d$ and  $q(K'')< \min (q(K), q(K'))$. .
\item $|M(K'')|= \max(|M(K)|, |M(K')|)$ and $q(K'')=\min (q(K), q(K'))$. 
\end{enumerate}
See Lemma~\ref{lem:clump_joining} for the formal version of this. We think of the parameters $|M(K)|, q(K)$ of the clumps as satisfying a sort of ``conservation law'' --- while the cut-density parameter $q(K)$ can decrease when joining them this can \emph{only happen whilst simultaneously increasing $|M(K)|$ via (1)}. The punchline to the proof is that there is an absolute bound to how many times (1) can happen --- if it happens more than $3\varepsilon^{-1}$ times, then we get a clump containing a regular subgraph larger that $(2+\varepsilon)d$, and Lemma~\ref{lem:main_tree_embedding_lemma} (i) gives a copy of $T$. Thus, although there can be an unlimited number of joining operations involving any clump, all but a constant number of them must be of type (2). Since there is no loss of cut-density $q(K)$ in (2), this allows us to deduce that the final clumps $K_1, \dots, K_t$ still maintain some connectivity and return to the proof strategy outlined above.

\section{Preliminaries}
For a graph $G$ and a set of vertices $S$, we use $G\setminus S$ to   denote $G$ with the vertices of $S$ deleted. For a set of edges $T$,  we use $G\setminus T$ to   denote $G$ with the edges of $T$ deleted.
In particular two graphs $G,H$, we use $G\setminus V(H)$ to denote $G$ with the vertices of $H$ deleted, $G\setminus E(H)$ to denote $G$ with the edges of $H$ deleted, $G\cap H$ to denote the graph with  edge set $E(G)\cap E(H)$ and isolated vertices deleted.
We use $G\setminus H$ to denote the graph with edge set $E(G)\setminus E(H)$ and isolated vertices deleted.
We use $G\cup H$ to denote the graph on $V(G)\cup V(H)$ with edge set $E(G)\cup E(H)$.
For sets, $S,T$ in a graph $G$, we define $\delta_S(T):=\min_{v\in T} |N_G(v)\cap S|$.
Given a family $\mathcal F=\{G_1, \dots, G_t\}$ of graphs, we define the graph $\bigcup \mathcal F:=G_1\cup \dots \cup G_t$.

For two (possibly overlapping) sets of vertices $S,T$ in a graph $G$, let $e_G(S,T)$ denote the number of edges  in $G$ starting in $S$ and ending in $T$ (counting edges contained in $S\cap T$ only once). Note that for any $S,T$, we have $2e(S,T)\geq \sum_{s\in S} |N(s)\cap (T\cup S)|$, with equality when $S=T$. We will refer to this as the ``handshaking lemma'', since it is a slight variant of the usual lemma.
We denote the number of common neighbours $u$ and $v$ have in a graph $G$ by  $N_{common, G}(u,v):=N_G(u)\cap N_G(v)$. 

For non-negative $x,y$, we write $a=x\pm y$ to mean $x-y\leq a\leq x+y$. Note that for $\varepsilon \in (0,1/2)$, if $x=(1\pm \varepsilon)y$, then $y=(1\pm 2\varepsilon)x$ ($x=(1\pm \varepsilon)y$ gives $y/(1+2\varepsilon)\leq (1-\varepsilon)y\leq x\leq (1+\varepsilon )y\leq y/(1-2\varepsilon)$, which rearranges into $x(1-2\varepsilon)\leq y\leq x(1+2\varepsilon)$).  
 We use standard notation for hierarchies of constants, writing $x\gg y$ to mean that there is a   function $f : (0,1] \rightarrow (0, 1]$ such that all subsequent statements hold for $x\geq f(y)$.

A $d$-ary tree $T$ is one in which every non-leaf vertex has degree $d$. A perfect $d$-ary tree of height $h$ is one where there is a designated root $r$, and every leaf is at distance $h$ from $r$.
A $1$-subdivision of a graph $G$ is a graph formed by replacing every edge with a path of length $2$ between its endpoints (with the middle vertices of these paths distinct new vertices not present in $G$).
An embedding of a graph $G$ into a graph $H$ is an injective homomorphism $f:G\to H$.
\subsection{Basic results}
In many contexts, the following two standard results allow one to reduce from a graph with large average degree to a connected graph with large minimum degree
\begin{fact}\label{fact:subgraph_min_degree_e/2}
Every graph $G$ with $e(G)\geq x|G|$ has a subgraph $H$ with $e(H)\geq x|H|$ and minimum degree $\delta(H)\geq \lfloor x\rfloor+1\geq x$.
\end{fact}
\begin{proof}
Consider a minimal counterexample $G$. If  $\delta(G)\geq \lfloor x\rfloor+1$, we're done with $H=G$. So there is some vertex with $d(v)<\lfloor x\rfloor+1$. Since $d(v)$ is an integer, $d(v)\leq \lfloor x\rfloor$. Set $G'=G-v$, noting 
$e(G')=e(G)-d(v)\geq e(G)-\lfloor x\rfloor\geq x|G|-\lfloor x\rfloor=x(|G'|+1)-\lfloor x\rfloor=x|G'|+x-\lfloor x\rfloor\geq x|G'|$. This would mean that $G'$ is a smaller counterexample than $G$. 
\end{proof}

\begin{fact}\label{fact:connectected_component_with_same_density_as_G}
Every graph $G$ has a connected component $C$ with $e(C)/|C|\geq e(G)/|G|$.
\end{fact}
\begin{proof}
Suppose not. Letting $C_1, \dots, C_t$ be the connected components of $G$, we have $e(G)=e(C_1)+\dots+e(C_t)< e(G)|C_1|/|G|+\dots+e(G)|C_t|/|G|=e(G)(|C_1|+\dots+|C_t|)/|G|=e(G)$, which is a contradiction.
\end{proof}

Minimum degree $\delta(G)\geq |T|$ allows one to greedily find an embedding of $T$ into $G$. This simple fact is surprisingly important in this paper --- almost all our tree embedding results come down to applying this to a suitably defined $G$.
\begin{fact}[Greedy tree embedding]\label{fact:greedy_tree_embedding}
Let $G$ be a graph, $v\in V(G)$ with $\delta(G-v)\geq d$ and $d(v)\geq \Delta$. Let $T$ be  $\leq d$-edge  tree with $\Delta(T)\leq \Delta$. Then $G$ has a copy of $T$ rooted at $v$.
\end{fact}
\begin{proof}
Let $r$ be the root of $T$, $x_1, \dots, x_{d_T(r)}$ the neighbours of $r$, and $y_1, \dots, y_m$ the remaining vertices of $T$. We can suppose that these are ordered so that $r, x_1, \dots, x_{d_T(r)}, y_1, \dots, y_m$ is a degeneracy ordering. Embed $r$ to $v$, and $x_1, \dots, x_{d_T(r)}$ inside $N_G(v)$ (which we can do because $d_G(v)\geq \Delta\geq d_T(r)$). Afterwards embed $y_1, \dots, y_m$ one by one. If $y_iy_j$ is the edge in $T$ with $j<i$, then when embedding $y_i$, embed it to an unused neighbour of $y_j$ (which we can do because this vertex has degree $\geq d$).  
\end{proof}

Our proof is not primarily probabilistic --- but we do use concentration bounds for some auxiliary lemmas.	 
\begin{lemma}[Chernoff bounds, \cite{molloy2002graph}]Given a binomially distributed variable $X\in Bin(n, p)$  for all  $0<a\leq 3/2$ we have
 $$\Prob{[|X- \Exp[X]|\geq a \Exp[X]]}\leq 2e^{-\frac{a^2}{3}\Exp[X]}$$.
\end{lemma}

	Given a product space $\Omega=\prod_{i=1}^n\Omega_i$ and a random variable $X:\Omega\to \mathbb{R}$ 
we  make the following definitions. If there is a constant $c$ such that changing $\omega\in \Omega$ in any one coordinate changes $X(\omega)$ by at most $c$. Then we say that $X$ is $c$-Lipschitz.
 \begin{lemma}[McDiarmid's Inequality, \cite{molloy2002graph}]\label{lem:McDiarmid}
For a product space $\Omega=\prod_{i=1}^n\Omega_i$ and a $c$-Lipschitz random variable $X:\Omega\to \mathbb{R}$, we have
$$\Prob\left(|X-\Exp(X)|>t \right)\leq 2e^{\frac{-t^2}{nc^2}}$$
\end{lemma}

\subsection{Cut-dense graphs}
Recall that a $n$-vertex graph $G$ is a $q$-cut-dense if for every partition $V(G)=A\cup B$, we have $e_G(A,B)\geq q|A||B|$. In this section, we go through the basics of this definition. 
We remark that $q$-cut-dense necessitates $q\leq 1$ since there can't be more than $|A||B|$ edges from $A$ to $B$.
We will often use that the definition is monotone in $q$, specifically  ``If $G$ is $p$-cut-dense, then it is also $q$-cut-dense for any $q\leq p$''. The ``dense'' part of the term ``cut-dense'' refers to the following.

\begin{fact} \label{fact:cut_dense_min_degree} 
Let $G$ be a $p$-cut-dense order $n\geq 10$ graph. Then $\delta(G)\geq p(n-1)\geq 0.9pn$ and $e(G)\geq pn^2/4$
\end{fact}
\begin{proof}
For a vertex $v\in V(G)$, the definition of $p$-cut-dense applied with $A=\{v\}, B=V(G)\setminus\{v\}$, gives $d(v)=e(v, V(G)\setminus\{v\})\geq p\cdot 1\cdot (n-1)$. Together with $n\geq 10$, this implies $\delta(G)\geq p(n-1)\geq p(n-n/10)= 0.9pn$ and $e(G)\geq n\delta(G)/2\geq pn^2/4$.
\end{proof}

Cut-density is preserved by deleting a few vertices from $G$.
\begin{lemma} \label{lem:cut_dense_delete_small_set}
Let $G$ be an order $n$, $q$-cut-dense graph. For any $U\subseteq V(G)$, with $|U|\leq qn/8$, $G\setminus U$ is $q/2$-cut-dense.
\end{lemma}
\begin{proof}
Consider a partition $V(G\setminus U)=A\cup B$. Without loss of generality these are labelled so that $|A|\geq |B|$, giving $|A|\geq \frac12(n-n/8)\geq n/4$. Since $A\cup U, B$ partition $V(G)$, the definition of $q$-cut-dense gives $e_G(A\cup U, B)\geq q|A\cup U||B|\geq q|A||B|$. Also $e_G(U,B)\leq |U||B|\leq qn|B|/8\leq q|A||B|/2$. Thus $e_{G\setminus U}(A,B)=e_G(A,B)=e_G(A\cup U, B)-e_G(U,B)\geq q|A||B|-q|A||B|/2=q|A||B|/2$.
\end{proof}

In the following two lemmas we often use the set-theoretic fact that whenever $X\subseteq Y\cup Z$, we have $|Y\cap X|+|Z\cap X|\geq |Y\cap X|+|Z\cap X|-|X\cap Y\cap Z|= |X|$ (the last equation is the inclusion/exclusion principle). The following   lemma is absolutely crucial to this paper --- it is basically the property $(\ast)$ referred to in the proof overview.

\begin{lemma} \label{lem:cut_dense_union} 
 Let $G_1,G_2$ be  $q$-cut-dense. Then $G_1\cup G_2$ is $\frac{q|V(G_1)\cap V(G_2)|}{4|V(G_1)\cup V(G_2)|}$-cut-dense.
\end{lemma}
\begin{proof}
Let $I:=V(G_1)\cap V(G_2)$. Consider a partition $V(G_1\cup G_2)=A_1\cup A_2$.
Since $I\subseteq A_1\cup A_2$ we have $|I\cap A_1|+|I\cap A_2|\geq |I|$, and so for some $i\in \{1,2\}$, we have $|I\cap A_i|\geq |I|/2$. Let $j$ be the number in $\{1,2\}\setminus i$. Since $A_j\subseteq V(G_1)\cup V(G_2)$, we have $|A_j\cap V(G_1)|+|A_j\cap V(G_2)|\geq |A_j|$, and so  for some $k\in \{1,2\}$, we have $|A_j\cap V(G_k)|\geq |A_j|/2$. Using that $G_k$ is $q$-cut-dense, $e_{G_1\cup G_2}(A_i,A_j)\geq e_{G_k}(V(G_k)\cap A_i, V(G_k)\cap A_j)\geq q|A_i\cap V(G_k)||A_j\cap V(G_k)|\geq q|I\cap A_i||A_j\cap V(G_k)|\geq q|I||A_j|/4=\frac{q|V(G_1)\cap V(G_2)|}{4|V(G_1)\cup V(G_2)|}|V(G_1)\cup V(G_2)||A_j|\geq \frac{q|V(G_1)\cap V(G_2)|}{4|V(G_1)\cup V(G_2)|}|A_i||A_j|$ as required.  
\end{proof}
 
The next lemma gives another equally important way of getting a larger cut-dense graph from a small one. 
\begin{lemma} \label{lem:cut_dense_add_vertices} 
Let $\delta \leq 1$.
Let $G$ be  $q$-cut-dense and $H$ a graph containing $G$ such that  $|N_H(v)\cap V(G)|\geq \delta|G|$ for all $v\in V(H)\setminus V(G)$. 
 Then $H$ is   $(\frac{q\delta|G|^2}{4|H|^2})$-cut-dense.
\end{lemma}
\begin{proof}
Consider a partition $V(H)=A_1\cup A_2$. Since $V(G)\subseteq A_1\cup A_2$, we have $|A_1\cap V(G)|+|A_2\cap V(G)|\geq |G|$, and so for some $i\in \{1,2\}$, we have $|A_i\cap V(G)|\geq |G|/2$. Let $j$ be the number in $\{1,2\}\setminus i$.
If $|A_j\cap V(G)|\geq \frac{\delta|G|}{2|H|}|A_j|$, then ``$G$ $q$-cut-dense'' gives $e_H(A_i,A_j)\geq e_G(A_i\cap V(G), A_j\cap V(G))\geq q |A_i\cap V(G)||A_j\cap V(G)|\geq \frac{q\delta|G|}{4|H|}|G||A_j|=(\frac{q\delta|G|^2}{4|H|^2})|H||A_j|\geq (\frac{q\delta|G|^2}{4|H|^2})|A_i||A_j|$. Thus we can assume $\frac{\delta|G|}{2|H|}|A_j|>|A_j\cap V(G)|$, which implies $|A_j\setminus V(G)|=|A_j|-|A_j\cap V(G)|> (1-\frac{\delta|G|}{2|H|})|A_j|\geq |A_j|/2$  and $\delta|G|/2=\frac{\delta|G|}{2|H|}|H|\geq \frac{\delta|G|}{2|H|}|A_j|>|A_j\cap V(G)|= |V(G)\setminus A_i|$. We have $e_H(A_j, A_i)\geq e_H(A_j\setminus V(G), A_i\cap V(G))= \sum_{v\in A_j\setminus V(G)} |N_H(v)\cap A_i\cap V(G)|= \sum_{v\in A_j\setminus V(G)}(|N_H(v)\cap V(G)|-|N_H(v)\cap V(G)\setminus A_i|)\geq \sum_{v\in A_j\setminus V(G)}(|N_H(v)\cap V(G)|-|V(G)\setminus A_i|)\geq |A_j\setminus V(G)|(\delta_{V(G)}(V(H)\setminus V(G))-|V(G)\setminus A_i|)\geq (|A_j|/2)(\delta|G|-\delta|G|/2)=\frac{\delta|G|}{4|H|}|A_j||H|\geq \frac{\delta|G|}{4|H|}|A_j||A_i|\geq \frac{q\delta|G|^2}{4|H|^2}|A_j||A_i|$ (for the last inequality we used that in any $q$-cut-dense graph, $q\leq 1$).
\end{proof}

Cut-density has various equivalent notions (like the ``connected cluster graph'' one referred to in the proof sketch). The only one we use is the following one which could be thought as a variant of edge-connectedness.
\begin{lemma}\label{lem:cut_dense_characterization_paths}
Let $n\geq 10$ and $G$ an $n$-vertex graph.
\begin{enumerate}[(i)]
\item  Suppose that $G$ is $q$-cut-dense. Then for every $u,v\in V(G)$, there are $\geq q^3n^2/9$ edge-disjoint length $\geq 1$ paths from $N(u)$ to $N(v)$.
\item Suppose that for every $u,v\in V(G)$, there are $\geq qn^2$ edge-disjoint length $\geq 1$ paths from $N(u)$ to $N(v)$. Then $G$ is $(q/10)$-cut-dense.
\end{enumerate}
\end{lemma}
\begin{proof}
(i):  Let $N_u\subseteq N(u), N_v\subseteq N(v)$ be disjoint subsets of order $qn/3$ (which we can find using Fact~\ref{fact:cut_dense_min_degree}). If there are $q^3n^2/9$ edge-disjoint paths from $N_u$ to $N_v$, then we are done --- so suppose this doesn't happen.  Then Menger's Theorem gives us a set $S$ of $<  q^3n^2/9$ edges that separates $N_u,  N_v$ i.e.  $G\setminus S$ has no paths from $N_u$ to $N_v$. Let $A$ be the union of the connected components of $G\setminus S$ intersecting $N_u$, $B:=V(G)\setminus A$, noting that ``$G\setminus S$ has no paths from $N_u$ to $N_v$'' is equivalent to $N_v\subseteq B$. Note $|A|\geq |N_u|= qn/3, |B|\geq |N_v|= qn/3$. We have $E_G(A,B)\subseteq S$, giving $e_G(A,B)<q^3n^2/9\leq q|A||B|$, contradicting $q$-cut-density.

(ii): First note that $\delta(G)\geq qn$ --- indeed otherwise, if there was a vertex with $|N(v)|< qn$, the number of edge-disjoint paths from $N(v)$ to $N(v)$ would be bounded above by $e(N(v), V(G))\geq n|N(v)|<qn^2$ (since all the starting edges of these paths must be distinct).
Now consider a partition $V(G)=A\cup  B$. 
Suppose that every $a\in A$ has $|N(a)\cap B|\geq qn/10$. Then $e(A,B)=\sum_{a\in A}|N(a)\cap B|\geq |A|qn/10\geq q|A||B|/10$.
The same argument works if   every $b\in B$ has $|N(b)\cap A|\geq qn/10$.

Thus we can suppose that there exist $a\in A, b\in B$ with $|N(a)\cap B|, |N(b)\cap A|\leq qn/10$. This gives $|N(a)\cap A|\geq d(a)-qn/10\geq 0.9d(a)$ and $|N(b)\cap B|\geq d(b)-qn/10\geq 0.9d(b)$. Let $\mathcal P$ be a set of $qn^2$ edge-disjoint paths from $N(a)$ to $N(b)$. Let $\mathcal P'$ be the subset of these paths touching $(N(a)\cap B)\cup (N(b)\cap A)$. Since these paths must use distinct edges through $(N(a)\cap B)\cup (N(b)\cap A)$, we have $|\mathcal P'|\leq e((N(a)\cap B)\cup (N(b)\cap A), V(G))\leq |N(a)\cap B|n + |N(b)\cap A|n\leq qn^2/5$.   All paths in $\mathcal P\setminus \mathcal P'$ start in $N(a)\setminus (N(a)\cap B)\subseteq A$ and end in $N(b)\setminus (N(b)\cap A)\subseteq B$, and so all paths in $\mathcal P\setminus \mathcal P'$ have an edge of $G[A,B]$. This gives $e(A,B)\geq| \mathcal P\setminus \mathcal P'|\geq qn^2-qn^2/5\geq qn^2/10$ as required.
\end{proof}
Cut density is conveniently preserved by taking random subgraphs.
\begin{lemma}\label{lem:cut_dense_random_sample}
Let $p,q\gg n^{-1}$
Let $G$ be a $n$-vertex, $q$-cut-dense, $S$ a  $p$-random subset. With high probability   $G[S]$ is $p^{20q^{-3}} q^3/400$-cut-dense.
\end{lemma}
\begin{proof}
By Lemma~\ref{lem:cut_dense_characterization_paths} (i),  for every $u,v\in V(G)$, there is a family $\mathcal P_{u,v}$ of $\geq q^3n^2/9$ edge-disjoint length $\geq 1$ paths from $N(u)$ to $N(v)$. Note that at most $e(G)/(20q^{-3})< q^3n^2/20$ of these can have length $\geq 20q^{-3}$, so letting $\mathcal P_{u,v}'$ be the paths in $\mathcal P_{u,v}$ shorter  than $20q^{-3}$, we have $|\mathcal P_{u,v}'|\geq q^3 n^2/9-q^3n^2/20>q^3n^2/20$. For every path $P$, we have $P(P \subseteq G[S])=p^{|P|}$. Letting $\mathcal Q_{u,v}$ be the paths from $\mathcal P_{u,v}$ that are in $G[S]$, we get $\Exp(|\mathcal Q_{u,v}|)=\sum_{P\in \mathcal P_{u,v}} p^{|P|}\geq \sum_{P\in \mathcal P'_{u,v}} p^{|P|}\geq \geq \sum_{P\in \mathcal P'_{u,v}} p^{20q^{-3}} = p^{20q^{-3}}|\mathcal P_{u,v}'| \geq  p^{20q^{-3}} q^3n^2/20$. Note that since the paths in $\mathcal P_{u,v}$ are edge-disjoint, there can be $\leq d(v)\leq n$ of them through any vertex, and so $\mathcal Q_{u,v}$ is $n$-Lipschitz. By McDiarmid's Inequality, we have that $|\mathcal Q_{u,v}|< p^{20q^{-3}} q^3n^2/40$ with probability $\leq 2\exp(- (p^{20q^{-3}} q^3n^2/40)^2/n^3)<o(n^{-2})$.  Taking a union bound over all $u,v$, we have that with high probability, $|\mathcal Q_{u,v}|\geq p^{20q^{-3}} q^3n^2/40$ for all $u,v$. By Lemma~\ref{lem:cut_dense_characterization_paths} (ii), $G[S]$ is $p^{20q^{-3}} q^3n/400$-cut-dense.
\end{proof}
 The following lemma extends the previous one: not only does taking a random induced subgraph preserve cut-density --- but the random sample even functions as a ``cut-density core'' (meaning that adding any vertices to the random sample also preserves cut-density).
\begin{lemma} \label{lem:cut_dense_random_sample_plus_more} 
Let $1\gg p,q\gg n^{-1}$.
Let $G$ be a $n$-vertex, $q$-cut-dense, $S$ a $p$-random subset. With high probability  every induced subgraph of $G$ containing $S$ is $p^{q^{-4}}q^5$-cut-dense.
\end{lemma}
\begin{proof}
By Chernoff's bound and Fact~\ref{fact:cut_dense_min_degree}, with high probability, every vertex has $P(|N(v)\cap S|< qpn/4)< 2e^{-qpn/12}$, and also $P(|S|< pn/2)< e^{-pn/12}$. By a union bound, with high probability every vertex has $|N(v)\cap S|\geq qpn/4$, $|S|\geq pn/2$, and $G[S]$ is $p^{20q^{-3}} q^3/400$-cut-dense (using Lemma~\ref{lem:cut_dense_random_sample}). Now considering any induced subgraph $H$ containing $S$, Lemma~\ref{lem:cut_dense_add_vertices} tells us that $H$ is $\frac{(p^{20q^{-3}} q^3/400)(qp/4)(pn/2)^2}{4n^2}=\frac{p^{20q^{-3}+3} q^4}{12800}\geq   p^{q^{-4}} q^5$-cut-dense. 
\end{proof}
It's very easy to approximately decompose any graph into cut-dense components.
\begin{lemma}\label{lem:cut_dense_decomposition} 
Let $G$ be a graph. It is possible to delete $q n^2$ edges from $G$ to get a subgraph all of whose components are $q$-cut-dense.
\end{lemma}
\begin{proof}
Repeat the following process: as long as there is a connected component $C$ and a partition $V(C)=A\cup B$ with $e_G(A,B)<q|A||B|$, delete all edges between $A$ and $B$. Call the resulting graph $H$. It is immediate that components in $H$ are $q$-cut-dense.
 
Consider a vertex $v$. Let $(A_1, B_1), \dots, (A_k, B_k)$ be the pairs between which edges were deleted throughout the process with $v\in A_1\cup B_1, \dots, A_k\cup B_k$. Without loss of generality, suppose that things are labelled so that $v\in A_1, \dots, A_k$. Note that for all $i$, $B_1, \dots, B_i, A_i$ are disjoint (this is proved by induction on $i$ with the initial case  true because all $A_i, B_i$ are chosen to be disjoint. For the induction step, note that after edges between $A_i, B_i$ are deleted, the connected component containing $v$ is contained inside $A_i$, and so $A_{i+1}\cup B_{i+1}\subseteq A_i$ which is disjoint from $B_1, \dots, B_{i}$ by induction).
 Let $f(v)=\sum e_G(A_i, B_i)/|A_i|\leq \sum q|B_i|\leq qn$ (using that $B_1, \dots, B_k$ are disjoint for the last inequality). 
 
 Let $\mathcal A$ be the set of all pairs $(A,B)$ used throughout the process. 
 We have \begin{align*}
 qn^2&\geq \sum_{v\in V(G)} f(v)=\sum_{v\in V(G)}\left(\sum_{(A,B)\in \mathcal A,  A\ni v} e_G(A, B)/|A|+\sum_{(A,B)\in \mathcal A,  B\ni v} e_G(A, B)/|B|\right)\\
 &=\sum_{(A,B)\in \mathcal A}\left(\sum_{ v\in A} e_G(A, B)/|A|+\sum_{ v\in B} e_G(A, B)/|B|\right)=\sum_{(A,B)\in \mathcal A}2e_G(A,B)=2e(G\setminus H)
 \end{align*}
\end{proof}

A consequence of the above is that every dense graph has large cut-dense subgraphs.
\begin{lemma}\label{lem:cut_dense_subgraph_in_dense_graph}
Let $1\gg p\gg q\gg \mu \gg n^{-1}$.
Let $G$ have $n$ vertices and $e(G)\geq pn^2$. Then $G$ has  a $q$-cut-dense subgraph of order $\mu n$.
\end{lemma}
\begin{proof} 
Let $1\gg p\gg s \gg q\gg \mu \gg n^{-1}$.
Let $G'$ be an induced subgraph formed by keeping every vertex with probability $2s^{-1}\mu$.   By McDiarmid's Inequality, with high probability  we have $|G'|\in[s^{-1}\mu n, 4s^{-1}\mu n]$ and $e(G')\geq   (2s^{-1}\mu)^2 (pn^2)/2\geq p|G'|^2/32$. 
Use Lemma~\ref{lem:cut_dense_decomposition} to delete $s|G'|^2$ edges in order to get a graph $G''$ with $\geq p|G'|^2/32-s|G'|^2\geq p|G'|^2/64=p|G''|^2/64$ edges  whose components are $s$-cut-dense. Since $\Delta(G'')\geq 2e(G'')/|G''|\geq p|G''|/64$, some component $C$ of $G''$ has $|C|\geq p|G''|/64=p|G'|/64\geq p(s^{-1}\mu n)/64\geq \mu n$ vertices. Let $C'$ be a random subset of $V(C)$ formed by keeping every vertex with probability $\alpha:=\frac{\mu n}{2|C|}$ noting $s/8=\frac{\mu n} {2(4s^{-1}\mu n)}\leq \frac{\mu n}{2|G'|} \leq \alpha\leq \frac{\mu n}{2p(s^{-1}\mu n)/64}= 32p^{-1}s$. With high probability, $C'$ satisfies Lemma~\ref{lem:cut_dense_random_sample_plus_more} and has $|C'|\leq 3\mu n/4$ (using Chernoff's bound for the latter). Now letting $H$ be an induced subgraph of $C$ of order $\mu n$ containing $C'$ gives a graph which is $\alpha^{s^{-4}}s^5\geq q$-cut-dense. 
\end{proof}

When a high minimum degree graph contains no cut-dense subgraphs, then that graph must expand. As mentioned in the introduction, this is how everything is reduced to the ``locally dense'' case.
\begin{lemma}\label{lem:expansion_no_cut_dense}
Let $\Delta^{-1}, \varepsilon\gg p\gg \mu \gg d^{-1}$.
Let $G$ be a graph with  $\delta(G)\geq \varepsilon d$ and containing no $p$-cut-dense graph of order $\mu d$. Then every set $S\subseteq V(G)$ with $|S|\leq 10d$ has $|N(S)|\geq 10\Delta|S|$.
\end{lemma}
\begin{proof}
Suppose for contradiction that there is some $S$ with $|S|\leq 10d$ and $|N(S)|< 10\Delta|S|$.  This gives $|S|> |S|/2+|N(S)|/20\Delta\geq (|S|+|N(S)|)/20\Delta \geq |S\cup N(S)|/20\Delta$.
Then $e(S\cup N(S))\geq \sum_{s\in S} d(s)/2\geq \varepsilon d|S|/2\geq\varepsilon |S|^2/20 \geq \varepsilon |S\cup N(S)|^2/8000\Delta^2$ (the first inequality comes from the handshaking lemma, the second inequality is $\delta(G)\geq \varepsilon d$, the third inequality is $|S|\leq 10d$, the fourth inequality is  $|S|>|S\cup N(S)|/20\Delta$). Note that $|N(S)|< 10\Delta|S|$ tells us that $S\neq \emptyset$, and so $|S\cup N(S)|\geq |N(S)|\geq \delta(G)\geq \varepsilon d$.  
Lemma~\ref{lem:cut_dense_subgraph_in_dense_graph} (applied with $p=\varepsilon/8000\Delta^2$, $q=p$, $\mu=\mu$, $n=|S\cup N(S)|\geq \varepsilon d$) gives us a $p$-cut-dense graph of order $\mu d$. 
\end{proof}

 \section{Clumps}
The following is the key technical definition that makes the proof work. 
\begin{definition}\label{def:clump}
A clump $K$ with parameters $p,\kappa, m$ is a graph (whose vertices and edges we denote by $V(K)$, $E(K)$) together with  $\mathcal H(K), \mathcal M(K), C(K), k(K)$ that satisfy the following:
\begin{enumerate} 
\item[{(H):}] $\mathcal H(K)=\{H_1, \dots, H_h\}$ is a collection of edge-disjoint order $m$ $p$-cut-dense graphs which partition $E(K)$. 
\item[{(M):}] $\mathcal M(K)=\{M_1, \dots, M_{k(K)}\}$ is a  collection of vertex-disjoint  $p^{13} m$-regular graphs, each contained in some $H_i\in \mathcal H(K)$. Additionally $k(K)$ is as large as possible.
\item[{(C):}] $C(K)$ is a $\kappa^{(10k(K))!}$-cut-dense, $|C(K)|\leq 4^{k(K)}m$, and $\bigcup \mathcal M(K)\subseteq C(K)$.
\end{enumerate}
\end{definition}
We emphasise that by ``$k(K)$ is as large as possible'' we formally mean ``there is no collection $\mathcal M'$ of $k(K)+1$ vertex-disjoint $p^{13} m$-regular graphs, each contained in some $H_i\in \mathcal H(K)$''. The number $h$ can be any non-negative integer and won't be relevant in proofs (which is why we don't define it to be a parameter of the clump like $p, \kappa, m$). Comparing this definition to the proof overview, notice that we no longer have a special parameter $q(K)$ tracking cut-density --- instead $k(K)$ simultaneously controls all parameters which might vary between different clumps. The parameters $p,\kappa, m$ should be thought of as being constant for the whole proof.

Set $M(K):=\bigcup \mathcal M(K)$.
Let $B(K):=M(K)\cup\{uv:u\in M(K), v\not\in M(K), |N_K(v)\cap V(M(K))|\geq pm/2\}$. The following is where the small covers in Theorem~\ref{main_theorem} ultimately come from.
\begin{fact}\label{fact:BK_small_covers}
For any clump,  $M(K)$ is a cover of $B(K)$.
\end{fact}
\begin{proof}
By definition, the edges of $B(K)$ are either in $M(K)$ or of the form $uv$ with $u\in M(K)$.
\end{proof}
The order of $M(K)$ increases with $k(K)$.
\begin{fact}\label{fact:MK_lower_bound_by_k}
For any clump $k(K)\le |M(K)|/p^{13}m$
\end{fact}
\begin{proof}
Since $p^{13}m$-regular graphs have order $\ge p^{13}m$, the union of the $k(K)$ many vertex disjoint $p^{13}m$-regular graphs in $\mathcal M(K)$ has order $|M(K)|=|\bigcup \mathcal M(K)|\geq k(K)p^{13}m$.
\end{proof}

Let $D(K):=C(K)\cup B(K)=C(K)\cup\{uv:u\in M(K), v\not\in M(K), |N_K(v)\cap V(M(K))|\geq pm/2\}$ (using (C) for the last equation).
\begin{lemma}\label{lem:clump_lower_bound_eK}
Let $p\leq 1$.
For any clump $K$ with parameters $p, \kappa, m$,  we have $e(K)\geq p^{13}m|B(K)|/10$
\end{lemma}
\begin{proof}
Every vertex of $M(K)$ has degree $\geq p^{13}m\geq p^{13}m/2$ in $K$ by (M), and every vertex of $B(K)\setminus M(K)$ has degree $\geq pm/2\geq p^{13}m/2$ by definition of $B(K)$. The result follows from the handshaking lemma.
\end{proof}

We call a clump \emph{nonempty} if $k(K)\geq 1$. Though it won't be used in the proof, this is equivalent to asking that $e(K)>0$ or $\mathcal H(K)\neq \emptyset$. This can be seen via the following theorem.
\begin{theorem}[R\"odl-Wysocka, \cite{rodl1997note}]\label{thm:Rodl_Wosocka}
Let $1/2\geq \gamma\gg n^{-1}$.
Every graph with $e(G)\geq \gamma n^2$ contains a $\lceil 0.4\gamma^3 n\rceil$-regular subgraph.
\end{theorem}
We remark that in our proofs, this theorem could be replaced with a weaker one which produces a nearly-regular graph (e.g. one where all degrees are in $(1\pm \gamma)\gamma^3 n$). But this makes the numbers a bit more technical in later proofs, so we avoid doing this. 

The following lemma  shows that most of the edges in a clump are in $B(K)$. Since $B(K)$ is covered by $M(K)$, this is ultimately the source of the ``every component has a cover of size $\leq (2+\varepsilon)d$'' part of Theorem~\ref{main_theorem}. 
\begin{lemma}\label{lem:clump_large_intersection_with_B/D}
Let $m^{-1}\ll p\leq 1/100$.
Let $K$ be a clump with parameters $p,\kappa, m$ and $H\in \mathcal H(K)$.
Then $e(H\setminus B(K))\leq p^2m^2$ and $|V(H)\cap V(M(K))|\geq pm/2$. 
Additionally $e(K\setminus B(K))\leq 4pe(K)$.
\end{lemma}
\begin{proof}
Everything in this paragraph takes place inside the graph $H$. 
Note that   Fact~\ref{fact:cut_dense_min_degree} gives $\delta(H)\geq 0.9pm$. 
Let    $R_1=\{uv: u\in M(K), v\not\in M(K), |N(v)\cap V(M(K))|< pm/2\}$ and  $R_2=H\setminus V(M(K))$, noting that these edge-partition $H\setminus B(K)$. 
If $e(R_2)\geq p^4m^2$, then by Theorem~\ref{thm:Rodl_Wosocka}, it would contain a $p^{13}m$-regular subgraph $R_2'$, which we could add to $\mathcal M(K)$ to contradict maximality of $k(K)$ in (M). So we can assume $e(R_2)< p^4 m^2$. This shows that $|V(H)\cap V(M(K))|\geq pm/2$, as otherwise $e(R_2)=\frac12\sum_{v\in V(H)\setminus V(M(K))}|N_H(v)\setminus V(M(K))|\geq |V(H)\setminus V(M(K))|(\delta(H)-|V(H)\cap V(M(K))|)\geq (m-pm/2)(0.9pm-pm/2)\geq 0.2pm^2>p^4m^2$.
It also shows $|\{v\not\in M(K): |N(v)\cap V(M(K))|< pm/2\}|\leq p^2 m$ (since $\delta(H)\geq 0.9pm$, and so every vertex in this set touches $\geq0.4pm$ edges of $R_2$. This would give $e(R_2)\geq (p^2m)(0.4pm)/2=0.2p^3m^2$, contradicting $e(R_2)< p^4 m^2$).
This shows that  $e(R_1)< (pm/2)|\{v\not\in M(K): |N(v)\cap V(M(K))|< pm/2\}|\leq   p\times p^2m^2$, and so $e(H\setminus B(K))=e(R_1)+e(R_2)\leq p^3m^2+p^4m^2< p^2m^2$ as required. 

For the ``additionally'' part, note that since (H) tells us that $\mathcal H(K)$ is a partition of $E(K)$, we have $e(K\setminus B(K))=\sum_{H\in \mathcal H(K)}e(H\setminus B(K))\leq \sum_{H\in \mathcal H(K)}p^2m^2=p\sum_{H\in \mathcal H(K)} pm^2\leq p\sum_{H\in \mathcal H(K)} 4e(H)=4pe(K)$ (here the last inequality uses Fact~\ref{fact:cut_dense_min_degree}).
\end{proof}

The following important lemma is what we mean when we say that clumps are ``highly connected''.  To see why, consider specifically the the  ``$t=1$ case'' of the lemma --- this says that $C(K_1)$ together with an arbitrary   small subset $I_1\subseteq D_1$ and an arbitrary $H_1\in \mathcal H(K_1)$ are cut-dense. In other words --- all small, somewhat curated subgraphs of all clumps are cut-dense. Note also that the clumps are not required to be distinct/disjoint --- we will sometimes apply the lemma with  $K_1= \dots= K_t$. 
 \begin{lemma} \label{lem:clump_cut_dense_subgraphs}
 Let $m^{-1}\ll \mu, \kappa\leq p^{13}\leq 1/1000.$
Let $K_1, \dots, K_t$ be nonempty clumps with parameters $p,m,\kappa$ and let $k\geq \max k(K_i)$. Suppose we have $I_i\subseteq D(K_1)\cap D(K_i)$ with $|I_i|=\mu m$. Let $H_1, \dots, H_t$ such that for each $i$, $H_i=\emptyset$ or $H_i\in \mathcal H(K_i)$. Then $\bigcup (I_i\cup H_i\cup C(K_i))$ is $\kappa^{(10k)!}\mu^{2kt^2}$-cut-dense.
\end{lemma}
\begin{proof}
Note $\mu\leq p^{13}$ and $|C(K_i)|\geq p^{13} m$ (which comes from $M(K_i)\subseteq C(K_i)$, (M), and nonemptyness of $K$) implies that $\frac{|C(K_i)|}{|C(K_i)|+2\mu m}\in [1/4,1]$. By definition of $B(K_i), D(K_i)$, and $M(K_i)\subseteq C(K_i)$, all $v\in I_i\setminus C(K_i)\subseteq B(K_i)\setminus M(K_i)$ have $|N_{K_i}(v)\cap V(C(K_i))|\geq |N_{K_i}(v)\cap V(M(K_i))|\geq pm/2$. For the same reasons, all $v\in I_i\setminus C(K_1)$ have $|N_{K_1}(v)\cap V(C(K_1))|\geq pm/2$. 
By Lemma~\ref{lem:cut_dense_add_vertices}  and (C), $C(K_1)\cup I_1\cup I_i$, $C(K_i)\cup I_i$ are $\frac{\kappa^{(10k)!}}{4}(\frac{pm/2}{|C(K_i)|})(\frac{|C(K_i)|}{|C(K_i)|+2\mu m})^2\geq \kappa^{(10k)!}p4^{-4k}$-cut-dense. Lemma~\ref{lem:clump_large_intersection_with_B/D} gives $|V(C(K_i)\cup I_i)\cap V(H_i)|\geq |V(M(K_i))\cap V(H_i)|\geq pm/2$, $|V(C(K_1)\cup I_1\cup I_i)\cap V(H_1)|\geq |V(M(K_1))\cap V(H_1)|\geq pm/2$. 
By Lemma~\ref{lem:cut_dense_union}  $C(K_1)\cup H_1\cup I_1\cup I_i$, $C(K_i)\cup H_i\cup I_i$ are  $\frac14\kappa^{(10k)!}p4^{-4k}(\frac{p m/2}{4^km+m+2\mu m})\geq \kappa^{(10k)!}p^24^{-6k}$-cut-dense.    Applying  Lemma~\ref{lem:cut_dense_union} repeatedly, $\bigcup_{i=1}^{s}C(K_i)\cup H_i\cup I_i$ is $\kappa^{(10k)!}p^24^{-6k} \prod_{i=1}^s\frac{|I_i|}{4|\bigcup_{j=1}^i C(K_j)\cup H_j\cup I_j|}\geq \kappa^{(10k)!}p^24^{-6k} \prod_{i=1}^s\frac{\mu m}{4i(4^{k}m+ m+2\mu m)}\geq (\kappa^{(10k)!}p^24^{-6k})(\mu^{s}4^{-2ks}/s!)\geq \kappa^{(10k)!}\mu^{2ks^2}$-cut-dense. 
\end{proof}

At the start of our clumping process we give some initial cut-dense subgraphs the structure of clumps via the following.
\begin{lemma} \label{lem:clump_initialization}
Let $m^{-1}\ll\kappa\leq p\leq 1/1000$ and $m\in \mathbb{N}$. For every $p$-cut-dense order $m$ graph $H$, we  can define a clump $K$ with $C(K)=H, \mathcal H(K)=\{H\}$, with parameters $p,m, \kappa$, and $1\leq k(K)\leq p^{-{13}}$.
\end{lemma}
\begin{proof}
Let $\mathcal M(K)=\{M_1, \dots, M_{k(K)}\}$ be a collection of vertex-disjoint $p^{13}m$-regular subgraphs of $H$ with $k(K)$ as large as possible to get a family satisfying (M). We have $k(K)\geq 1$ by Theorem~\ref{thm:Rodl_Wosocka} (and $e(H)\geq pm^2/4$ using Fact~\ref{fact:cut_dense_min_degree}), and $k(K)\leq p^{-13}$ since $H$ doesn't have enough vertices to contain $>p^{-13}$ vertex-disjoint graphs of order $\geq p^{13}m$.
We have that $H$  is an order $m$, $p$-cut-dense graph giving (H).  By monotonicity and  ``$p\geq \kappa$'', $H$ is also $\kappa^{(10\cdot 1)!}\geq \kappa^{(10 k(K))!}$-cut-dense. Together with $|C(K)|=|H|=m\leq 4^1m\leq 4^{k(K)}m$ and $\bigcup \mathcal M(K)\subseteq H=C(K)$, this gives  (C).
\end{proof}

The following lemma is what drives our clumping process. It is the concrete version of the joining operations $(\ast)$, $(1)$, $(2)$ mentioned in the proof overview. 
\begin{lemma} \label{lem:clump_joining}
Fix $m^{-1}\ll \kappa\leq p\leq 1/1000$.
Let $K_1, K_2$ be edge-disjoint clumps with parameters $\kappa,p,m$ having $k(K_1), k(K_2)\le \kappa^{-1}$. 
Suppose $|V(D(K_1))\cap V(D(K_2))|\geq \kappa^{6(10\max(k(K_1), k(K_2)))!}m$.
Then there is a clump $K$ with $E(K)=E(K_1)\cup E(K_2)$, $\mathcal H(K)=\mathcal H(K_1)\cup \mathcal H(K_2)$ with parameters $\kappa,p,m$ and $\max(k(K_1), k(K_2))\leq k(K)\leq k(K_1)+k(K_2)$.
\end{lemma} 
\begin{proof} 
Without loss of generality $\max(k(K_1), k(K_2))=k(K_1)$. 
Note that $\mathcal H(K)$ is a family of order $m$, $p$-cut-dense graphs which partition $E(K)$ (just because $\mathcal H(K_1)$,  $\mathcal H(K_2)$ did the same for $E(K_1), E(K_2)$), and so $\mathcal H(K)$ satisfies (H).
 
Let $\mathcal N=\{N_1, \dots, N_k\}$ be a collection of vertex-disjoint $p^{13}m$-regular graphs, each contained in some $H_i\in \mathcal H(K_1)\cup \mathcal H(K_2)$, such that $k$ is as large as possible subject to this i.e ``there is no collection $\{N_1, \dots, N_{k+1}\}$ of vertex-disjoint $p^{13}m$-regular graphs, each contained in some $H_i\in \mathcal H(K_1)\cup \mathcal H(K_2)$''.
Note that we must have $k\geq k(K_1)$ (as otherwise $\mathcal M(K_1)$ would be a choice for $\mathcal N$ with larger $k$). We must also have $k\leq k(K_1)+k(K_2)\leq 2k(K_1)$ (as otherwise, since each $N_i$ is contained in $H\in \mathcal H(K_1)$ or $\mathcal H(K_2)$, we have either $|\{N_i: N_i\subseteq H\in \mathcal H(K_1)\}|> k(K_1)$ or $|\{N_i: N_i\subseteq H\in \mathcal H(K_2)\}|> k(K_2)$. These would contradict the maximality of $k(K_1)$/$k(K_2)$ in the definition of $K_1$/$K_2$ being clumps).
 We define $C(K),\mathcal M(K), k(K)$ satisfying (M), (C).
\begin{itemize}
\item If $k=k(K_1)$: Set $\mathcal M(K):=\mathcal M(K_1)$, $C(K):=C(K_1)$, and $k(K):=k(K_1)$. This satisfies  (C) just because $K_1$ was a clump. The ``$k(K)$ is as large as possible'' part of (M) holds because we've established ``there is no collection $\{N_1, \dots, N_{k+1}\}$ of vertex-disjoint $p^{13}m$-regular graphs, each contained in some $H_i\in \mathcal H(K_1)\cup \mathcal H(K_2)$'', while the other part of (M) holds just because $K_1$ was a clump.

\item If $k>k(K_1)$: Let $\mathcal M(K):=\mathcal N$ and $k(K):=k$, which satisfies (M) since $N$ was maximum. Let $I\subseteq V(D(K_1))\cap V(D(K_2))$ be a subset of order $\kappa^{6(10k(K_1))!}m$. For each $i=1, \dots, k$, we have some $H_i\in \mathcal H(K_1)\cup \mathcal H(K_2)$ with $N_i\subseteq H_i$.
We apply Lemma~\ref{lem:clump_cut_dense_subgraphs} with $t=k+2$, $\mu=\kappa^{6(10k(K_1))!}$, for $i=1, \dots, k$ each $K_i'=K_1$/$K_2$, depending on whether $H_i\in \mathcal H(K_1)$/$\mathcal H(K_2)$, $K_{k+1}'=K_1, K_{k+2}'=K_2, H_{k+1}=\emptyset, H_{k+2}=\emptyset$, and all $I_i=I$. This gives that $C:=C(K_1)\cup C(K_2)\cup I\cup \bigcup_{i=1}^k H_i$ is $\kappa^{(10k(K_1))!}(\kappa^{6(10k(K_1))!})^{2k(K_1)(k+2)^2}\geq \kappa^{(10k(K_1))!+24(10k(K_1))! k(K_1)^3}\geq \kappa^{(10k(K_1))!(10k(K_1))^5}\geq \kappa^{(10(k(K_1)+1))!}\geq \kappa^{(10k)!}=\kappa^{(10k(K))!}$-cut-dense.
Note that $|C|\leq |C(K_1)|+|C(K_2)|+|I|+km\leq (4^{k(K_1)}+4^{k(K_2)}+\kappa+k)m\leq 4\cdot 4^{k(K_1)}m\leq 4^{k(K)}m$ and $\bigcup \mathcal M(K)=\bigcup_{i=1}^kN_i\subseteq \bigcup_{i=1}^kH_i\subseteq C$, so we have checked (C).
\end{itemize} 
\end{proof}

We end this section by adapting Lemma~\ref{lem:main_tree_embedding_lemma} to the setting when we want to embed a tree into some clumps.
\begin{lemma}\label{lem:clump_tree_embedding}
Let $\Delta^{-1}, \varepsilon\gg p\gg h^{-1}\gg\kappa\gg L^{-1}\gg \gamma\gg d^{-1}$ and set $m:=d/h$. Set $C=2+\varepsilon$.
Let $K_1, \dots, K_t$ be edge-disjoint, nonempty clumps with parameters $p, \kappa, m$ in a graph $G$ so that $k(K_i)\leq 2Cp^{-13}h$ for all $i$ and $|V(D(K_i))\cap V(D(K_j))|\leq \kappa^{6(10\max(k(K_i), k(K_j))!}m$ for distinct $i,j$. Suppose that any of the following happen:
\begin{enumerate}[(1)]
\item $|M(K_i)|\geq (2+\varepsilon)d$ for some $i$.
\item For some $K_i$, there are $\geq Cp^{-3}h$ clumps $K_j$ with $|V(D(K_i))\cap V(D(K_j))|\geq \gamma m$.
\item We can define an auxiliary $L$-ary, depth $L$ tree $S$ with vertices $\{K_1, \dots, K_t\}$ and pick distinct vertices $u_{K_iK_j}\in V(D(K_i))\cap V(D(K_j))$ for all edges $K_iK_j\in E(S)$.
\end{enumerate}
Then $G$ has a copy of every $d$-vertex tree $T$ with $\Delta(T)\leq \Delta$.
\end{lemma}
\begin{proof}
\begin{enumerate}[(1)]
\item By (C), $C(K_i)$ is $\kappa^{10k(K_i)!}\geq \kappa^{10(2Cp^{-13}h)!}$-cut-dense, has order $|C(K_i)|\leq 4^{k(K_i)}m\leq 4^{2Cp^{-13}h}m$  and contains a $p^{13}m$-regular graph $M(K_i)$ of order $\geq (2+\varepsilon)d$. Thus Lemma~\ref{lem:main_tree_embedding_lemma} (i) applies (with $G=C(K_i)$, $n=|C(K_i)|$, $q=\kappa^{10(2Cp^{-13}h)!}, p'=p^{13}m/n\geq \frac{p^{13}}{4^{2Cp^{-13}h}}, \Delta=\Delta, \varepsilon=\varepsilon, d=d$, and no $L$ since it doesn't come up in (i)) to give a copy of $T$.

\item   
Without loss of generality, $t={Cp^{-3}h}$ and  the clumps are ordered so that  $|K_1\cap K_j|\geq \gamma m$, for $j=1, \dots t$ (for $i>1$ this is by assumption, while for $i=1$ we have $|K_1\cap K_1|=|K_1|\ge p^{13}m\gg \gamma m$ due to $K_1$ being nonempty).  For $i=1, \dots, t$, pick arbitrary order $m$, $p$-cut-dense $H_i\in \mathcal H(K_i)$, and pick order $\gamma m$ subsets $I_i\subseteq V(D(K_1))\cap V(D(K_i))$.
By Lemma~\ref{lem:clump_cut_dense_subgraphs} (with $k=2Chp^{-13}$, $\mu=\gamma, t=t$),  $\bigcup_{i=1}^t C(K_i)\cup H_i\cup I_i$ is $\kappa^{(2\cdot 10Chp^{-13})!}\gamma^{2(2Chp^{-13})t^2}\geq\gamma \cdot \gamma^{2(2Chp^{-13})(Cp^{-3}h)^2}\geq  \gamma^{h^4}$-cut-dense. 

Let $M$ be the largest $p^{13}m$-regular subgraph of $\bigcup_{i=1}^{t} H_i$. We claim that that $|V(M)\cap V(H_i\cap D(K_i))|\geq pm/20$ for all $i$. 
Indeed, suppose otherwise that for some $i$, $|V(M)\cap V(H_i\cap D(K_i))|< pm/20$. Recall Lemma~\ref{lem:clump_large_intersection_with_B/D} tells us that $e(H_i\cap D(K_i))=e(H_i)-e(H_i\setminus D(K_i))\overset{\footnotesize F~\ref{fact:cut_dense_min_degree}}\geq pm^2/4 -e(H_i\setminus D(K_i))\overset{{B(K_i)\subseteq D(K_i)}}{\geq} pm^2/4-e(H_i\setminus B(K_i)) \overset{L~\ref{lem:clump_large_intersection_with_B/D}}\geq (p/4-p^2)m^2$.
Then $e(H_i\cap D(K_i)\setminus V(M))\geq e(H_i\cap D(K_i))-|V(M)\cap V(H_i\cap D(K_i))||H_i|\geq (p/4-p^2)m^2-(pm/20)m\geq pm^2/8=p|V(H_i)|^2/8\geq p|H_i\cap D(K_i)\setminus V(M))|^2/8$. Set $n:=|H_i\cap D(K_i)\setminus V(M)|\geq \sqrt{e(H_i\cap D(K_i)\setminus V(M))}\geq \sqrt{p/8}m$, noting $n^{-1}\ll p/8\leq 1/2$.
 Theorem~\ref{thm:Rodl_Wosocka} applies to $H_i\cap D(K_i)\setminus V(M)$ in order to give a $p^{13}m$-regular subgraph disjoint from $M$. Adding it to $M$  contradicts maximality of $M$.

Thus $|M|\overset{incl/excl}{\geq} \sum_{i=1}^{t} |V(M)\cap V(H_i\cap D(K_i))|-\sum_{1\leq i< j\leq t}|V(M)\cap V(H_i\cap D(K_i))\cap V(H_j\cap D(K_j))|\geq t(pm/20) - \sum_{1\leq i< j\leq t}|V(D(K_i))\cap V(D(K_j))|\geq t(pm/20)-t^2\kappa m= Chp^{-3}(pm/2)-(Cp^{-3}h)^2\kappa m\geq Chp^{-1}m\geq Cd$.
Now Lemma~\ref{lem:main_tree_embedding_lemma} (applied to $G'=\bigcup_{i=1}^t C(K_i)\cup H_i\cup I_i$, $n=|\bigcup_{i=1}^t C(K_i)\cup H_i\cup I_i|\leq t(4^{10Chp^{-13}}+1+\gamma)m\leq 4^{20Chp^{-13}}m$, $p'=p^{13}m/n\geq p^{13}/4^{20Chp^{-13}}$, $q=\gamma^{h^4}$, $\Delta=\Delta, \varepsilon=\varepsilon, d=d$, and no $L$ since it doesn't come up in (i)) gives a copy of $T$.

\item 
If we are not done by (2), we can assume that  for all $K_i$, there are $<Cp^{-3}h$ clumps $K_j$ with $|V(D(K_i))\cap V(D(K_j))|\geq \gamma m$. Note that $t=|V(S)|\leq L^{2L}$.
Set $U:=\{u_{K_iK_j}: K_iK_j\in E(S)\}$.
For each $i$, let $U(K_i)=\{u_{K_iK_j}: K_j\in N_S(K_i)\}\subseteq D(K_i)$, noting $|U(K_i)|=d_S(K_i)\leq 2L \leq p^{13}m$.
Noting $|C(K_i)\cup U(K_i)|\overset{(C)}{\geq} |M(K_i)|\overset{(M)}{\geq} p^{13}m$, we can choose an order $p^{13}m$ set $I_i$ with $U(K_i)\subseteq I_i\subseteq C(K_i)\cup U(K_i)$.
 Let $K_i'$ be the induced subgraph of $D(K_i)$ on $V(K_i'):=C(K_i)\cup I_i=C(K_i)\cup U(K_i)\subseteq D(K_i)$, noting that
 $K_i'$ is $\kappa^{(10k(K_i))!}(p^{13})^{2\cdot 1^2\cdot k(K_i)}\geq \kappa^{2(10k(K_i))!}$-cut-dense by Lemma~\ref{lem:clump_cut_dense_subgraphs} (applied with $t=1$, $\mu=p^{13}$, $k=k(K_i), p=p, \kappa=\kappa, m=m$, $H_1=\emptyset$).
  
Let $B_i:= \bigcup_{K_j\neq K_i} V(D(K_j))\cap V(D(K_i))\setminus U(K_i)$, noting that $|B_i|\leq t\gamma m+(Cp^{-3}h)\kappa^{6(10k(K_i))!} m\leq (\gamma L^{2L} +(Cp^{-3}h)\kappa^{6(10k(K_i))!} )m\leq (L^{2L}\gamma +\kappa^{-1}\kappa^{6(10k(K_i))!} )m \leq 2\kappa^{-1}\kappa^{6(10k(K_i))!} m\leq  \kappa^{5(10k(K_i))!}) m\leq \kappa^{2(10k(K_i))!} p^{13}m/4$ (using ``$|V(K_i)\cap V(K_j)|\leq \kappa^{6(10\max(k(K_i), k(K_j))!}m$'' and the fact that we are not in (2) for the first inequality). 
Let $K_i''=K_i'\setminus B_i$, noting that it is $\kappa^{2(10k(K_i))!}/2\geq \kappa^{2(20Cp^{-13}h)!}/2$-cut-dense by Lemma~\ref{lem:cut_dense_delete_small_set} (with $n=|K_i'|\geq |M(K_i)|\geq p^{13} m$, $q=\kappa^{2(10k(K_i))!}$, $U'= B_i$, which has $|U'|\leq  qn/4$), and has $|K_i''|\geq |M(K_i)|-|B_i|\geq p^{13}m-\kappa^{2(10k(K_i))!} p^{13}m/4\geq p^{13}m/2=p^{13}d/2h\geq \frac{16 d}{(\kappa^{2(20Cp^{-13}h)!}/2)L}$.

Now let $S''$ be   the tree formed by replacing every $K_i$ in $S'$ by $K_i''$ (and replacing each edge $K_iK_j$ by $K_i''K_j''$). We show that $S''$ satisfies the assumptions of Lemma~\ref{lem:main_tree_embedding_lemma} (ii)(with $q=\kappa^{2(20Cp^{-13}h)!}/2\gg L^{-1}$, $\Delta=\Delta$,  $L=L$, $d=d$, and no $p$ since it doesn't come up in (ii)), which gives us a copy of $T$. All the conditions have already been established in the paragraph above, except the last one  ``$V(K_i'')\cap V(K_j'')=\{u_{K_iK_j}\}$ for all edges $K_iK_j\in E(S)$ and $V(K_i'')\cap V(K_j'')=\emptyset$ for non-edges $K_iK_j\not\in E(S)$'' which we establish now: 

Consider some  $K_i'', K_j''$. 
Since $K_i''\subseteq K_i'\subseteq D(K_i), K_j''\subseteq K_i'\subseteq D(K_j)$, we have $V(K_i'')\cap V(K_j'')\subseteq V(D(K_i))\cap V(D(K_j))\subseteq (V(D(K_i))\cap V(D(K_j))\setminus U(K_i))\cup  (V(D(K_i))\cap V(D(K_j))\setminus U(K_j))\cup (U(K_i)\cap U(K_j))$. 
Since $D(K_i)\cap D(K_j)\setminus U(K_i)\subseteq B_i$ (which is disjoint from $K_i''$), and $D(K_i)\cap D(K_j)\setminus U(K_j)\subseteq B_j$ (which is disjoint from $K_j''$),  $K_i''$ and $K_j''$ can only intersect inside $U(K_i)\cap U(K_j)$. But $U(K_i)\cap U(K_j)$ either equals $\{u_{K_iK_j}\}$ (if $K_iK_j\in E(S')$) or equals $\emptyset$ (if $K_iK_j\not\in E(S')$), as required by Lemma~\ref{lem:main_tree_embedding_lemma}.  
\end{enumerate}
\end{proof}

\section{Proof of theorem}
As mentioned in the proof overview, towards the end of the proof of the theorem, we need to find a $1$-subdivision of a large $L$-ary height $L$ tree in a suitable defined auxiliary graph. This is done via the following. 
\begin{lemma}\label{lem:find_tree_1_subdivision}
Let $s\geq t$.
Let $G$ be a bipartite graph with parts $A,B$ with $\delta(A)\geq 2$ and $\delta(B)\geq d\geq 4096k^3t$ such that every $b\in B$ there are $\leq d/16s$ $b'\in B$ with $|N_{common}(b,b')|\geq t$, and no $b'\in B$ with $|N_{common}(b,b')|\geq s$. Then $G$ contains a $1$-subdivision of any order $\leq k$ tree $T$ with subdivided vertices in $A$.
\end{lemma}
\begin{proof}
Let $A_{small}=\{a\in A: d(a)< 8k\}$ and  $A_{large}=\{a\in A: d(a)\geq 8k\}$. 
If at least  half the edges of $G$ touch $A_{large}$, then 
$e(A_{large}, B)\geq e(G)/2=\sum_{b\in B}d(b)/2\geq \delta(B)|B|/2\geq d|B|/2\geq 8k|B|$ and $e(A_{large}, B)\geq \delta(A_{large})|A_{large}|\geq 8k|A_{large}|$ both hold, and hence $2e(A_{large}, B)\geq 8k|B|+8k|A_{large}|$, which rearranges to $e(A_{large}, B)/(|A_{large}|+|B|)\geq 4k$. Fact~\ref{fact:subgraph_min_degree_e/2} implies that $G[A_{large}, B]$ has a subgraph with minimum degree $2k$ and so we can find a copy of a $1$-subdivision of $T$ in that greedily (Fact~\ref{fact:greedy_tree_embedding}).

Thus at least half the edges of $G$ touch $A_{small}$. Let $G':=G[A_{small}, B]$, noting that $e(G')\geq e(G)/2\geq d|B|/2$, and also that $\delta_{G'}(A_{small})\geq 2$.
For a vertex $b$ in a bipartite graph $H$, let $N^2_H(b)=\{v\in V(H): |d_H(b,v)|=2\}=\{v\in V(H): |N_{common,H}(b,v)|\neq \emptyset\}$.
Partition $N^2_{G'}(b)=X(b)\cup Y(b)$ where $X(b)=\{b'\in N^2_{G'}(b): |N_{common, G'}(b,b')|\geq t\}$ and $Y(b)=\{b'\in N^2_{G'}(b): |N_{common, G'}(b,b')|< t\}$. Using the lemma's assumption, we have $|X(b)|\leq d/16s$. 
We have that $e(N_{G'}(b), N^2_{G'}(b))=\sum_{v\in N^2(b)}|N_{common,G'}(b,v)|=\sum_{v\in X(b)}|N_{common,G'}(b,v)|+\sum_{v\in Y(b)}|N_{common,G'}(b,v)|\leq |X(b)|s+|Y(b)|t\leq  (d/16s)s+|Y(b)|t$. On the other hand, since every vertex in $N_{G'}(b)\subseteq A_{small}$ has degree $\geq 2$, we have  $e(N_{G'}(b), N^2_{G'}(b))\geq |N_{G'}(b)|$. Combining gives $|N^2_{G'}(b)|\geq |Y(b)|\geq \frac{|N_{G'}(b)|-d/16}t$, and hence $\sum_{b\in B} |N^2_{G'}(b)|\geq \sum_{b\in B} \frac{|N_{G'}(b)|-d/16}t=\frac1t\sum_{b\in B}|N_{G'}(b)|-d|B|/16t=\frac1te(G')-d|B|/16t\geq \frac1{2t}e(G)-d|B|/16t\geq d|B|/2t-d|B|/16t\geq d|B|/16t$.

Let $H$ be the subgraph of $G'$ formed by randomly deleting all but $2$ of the edges through each $a\in A_{small}$, making the choices independently for each $a$. 
Note that for every length $2$ path $bab'$, we have $\Prob(b'\in N^2_H(b))\geq \Prob(ba, b'a\in E(H))=1/\binom{d(a)}2\geq 1/d(a)^2\geq 1/(8k)^2$.
By linearity of expectation, we have that for $b\in B$, we have $\Exp(|N^2_H(b)|)\geq |N^2_{G'}(b)|/64k^2$, and so $\Exp(\sum_{b\in B} |N^2_H(b)|)\geq\sum_{b\in B} |N^2_{G'}(b)|/64k^2 \geq (d|B|/16t)/64k^2= d|B|/1024tk^2$. Fix some outcome for which $\sum |N^2_H(b)|\geq d|B|/1024tk^2$.  Contract the vertices of $A_{small}$ in $H$ into edges (i.e. for each $a\in A_{small}$, delete $a$ and replace it with an edge $e_a$ joining the two vertices in $N_H(a)$) in order to get a multigraph, and then keep one copy of each multiedge to get a simple graph $S$. Note that for $b,b'\in B$, we have ``$bb'\in E(S)\iff b'\in N_{H}^2(b)$'', and hence $e(S)=\frac12 \sum |N^2_H(b)|\geq \frac12d|B|/1024tk^2=\frac12d|V(S)|/1024tk^2\geq 2k|V(S)|$. Using Fact~\ref{fact:subgraph_min_degree_e/2}, pass to a subgraph $S'$ of minimum degree $k$. Now greedily pick a copy  of $T$ in $S'$ (using Fact~\ref{fact:greedy_tree_embedding}). Uncontracting the vertices in $A$ gives us the $1$-subdivision we want. 
\end{proof}

 We now prove the main theorem, which we restate below in a slightly stronger form. To make the numbers a little neater, we changed the assumption ``$T$ has $d$ edges'' into ``$T$ has $d$ vertices''. We also exchange the assumption ``$e(G)\geq \varepsilon d|G|$'' for saying that the number of deleted edges is $\varepsilon dn$. We also changed the   sizes of the covers of the connected components in the conclusion from ``$\leq 3d$'' to the stronger ``$\le (2+\varepsilon)d$''. Applying the theorem below with $\varepsilon'=\varepsilon^2/2$, $d'=d+1$  implies Theorem~\ref{main_theorem} as stated in the introduction (there we have the assumption  $e(G)\geq \varepsilon d|G|$. Theorem~\ref{main_theorem_proof} deletes $\leq \varepsilon'd' |G|= \varepsilon^2 (d+1)|G|/2\leq \varepsilon^2 d|G|\leq \varepsilon e(G)$ edges, and leaves a graph whose components have covers of size $\leq (2+\varepsilon')d'=(2+\varepsilon^2/2)(1+\frac1d)d\leq (2+\varepsilon)d\le 3d$). 
\begin{theorem}\label{main_theorem_proof}
Let $1\geq \Delta^{-1}, \varepsilon\gg d^{-1}$.
Let $T$ be a tree with $d$ vertices and $\Delta(T)\leq \Delta$. For any $n$-vertex graph $G$ having no copies of $T$, it is possible to delete $\varepsilon dn$ edges to get a graph $H$ each of whose connected components has a cover of size $\leq (2+\varepsilon)d$.
\end{theorem}
\begin{proof}
Pick $\varepsilon, \Delta^{-1}, \gg p\gg  h^{-1}\gg \alpha\gg  \kappa\gg  L^{-1}\gg\gamma \gg d^{-1}$. Set $C:=2+\varepsilon\leq 3$.

 We can suppose that $e(G)\leq 2dn$, as otherwise $G$ has a subgraph with $\delta(G)\geq d$ (by Fact~\ref{fact:subgraph_min_degree_e/2}), and hence we greedily get a copy of $T$  (Fact~\ref{fact:greedy_tree_embedding}).
 Set $m:= d/h$, noting  $|T|=hm$\footnote{In general we obviously can't have both $m$ and $h$ being integers. The neatest way to formally deal with this, is to allow $h$ to be a non-integer, and to replace ``$h$'' by ``$\lceil h \rceil$'' at appropriate places in the proof that follows.}.
Pick a maximal collection of edge-disjoint order $m$ $p$-cut-dense subgraphs $K_1, \dots, K_t$ in $G$. Let $G_{dense}:=\bigcup K_i$, and set $G_{sparse}$ to be the graph on $V(G)$ with edges $E(G)\setminus E(G_{dense})$.
Suppose $e(G_{sparse})\geq \varepsilon dn/3$. Use Fact~\ref{fact:subgraph_min_degree_e/2} to pick a subgraph $G_{sparse}'\subseteq G_{sparse}$ of minimum degree $\varepsilon d/6$. Since $G_{sparse}'$ contains no order $m=d/h$ $p$-cut-dense subgraphs, Lemma~\ref{lem:expansion_no_cut_dense} (with $\varepsilon'=\varepsilon/6, \mu=h^{-1}$) tells us that ``all sets $S\subseteq V(G_{sparse}')$ with $|S|\leq 10d$ have $|N(S)|\geq 10\Delta |S|$''.  The Friedman-Pippenger Theorem~\ref{thm:Friedman-Pippenger} gives us a copy of $T$ in $G_{sparse}'$.  So we can suppose  $e(G_{sparse})< \varepsilon dn/3$.

Use Lemma~\ref{lem:clump_initialization} to give each $K_i$ in $G_{dense}'$ the structure of a nonempty clump with parameters $p, \kappa, m$. For as long as possible, repeat the following: if we have two clumps $K^-, K^+$ with $|M(K^-)|, |M(K^+)|\leq Cd$ 
and $|V(D(K^-))\cap V(D(K^+))|\geq \kappa^{6(10\max(k(K^-), k(K^+))!} m$, replace $K^-, K^+$ in our collection with the clump $K$ with $E(K)=E(K^-)\cup E(K^+)$ produced by Lemma~\ref{lem:clump_joining} (the property $k(K^{*})\le \kappa^{-1}$ holds for such clumps because Fact~\ref{fact:MK_lower_bound_by_k} gives $k(K^{*})\le |M(K^*)|/p^{13}m\le Cd/p^{13}m=hp^{-13}\ll \kappa^{-1}$). The process stops just because the total number of clumps keeps decreasing.
Let $\mathcal K$ be the final collection of clumps we get, noting that these are edge-disjoint, nonempty, and partition $G_{dense}$ (since the initial clumps had this property, and the property is maintained at each application of Lemma~\ref{lem:clump_joining}).

\begin{claim}\label{claim:M(Ki)_upper_bound}
$|M(K)|\leq Cd$ for all $K\in \mathcal K$.
\end{claim}
\begin{proof} 
First note that this holds for the initial clumps, since each $M(K_i)$ is contained in some $K_i$ which has order $m=d/h\leq Cd$. 
Now suppose for contradiction, that at some point in the joining process, we produced a clump $K$ with $|M(K)|>Cd$. Considering the earliest point when this happens, we have $K=K^-\cup K^+$ for some clumps $K^-, K^+$ with $|M(K^-)|, |M(K^+)|\leq Cd$. By (M), all clumps have $|M(K)|\geq p^{13}mk(K)$, which gives $k(K^-), k(K^+)\leq \frac{Cd}{p^{13}m}=Chp^{-13}$. From Lemma~\ref{lem:clump_joining}, $k(K)\leq k(K^-)+k(K^+)\leq 2Chp^{-13}$. Thus Lemma~\ref{lem:clump_tree_embedding} (1) applied with $t=1$ to the clump $K$ gives a copy of $T$
\end{proof}

This claim and (M) show that for all clumps $K\in \mathcal K$, we have $Cd\geq |M(K)|\geq p^{13}mk(K)$, which implies  $k(K)\leq Cd/p^{13}m=Chp^{-13}$. It also shows that each $B(K)$ has a cover of order $\leq Cd$, namely $M(K)$ (Fact~\ref{fact:BK_small_covers}).
 By the assumption that the process stopped and Claim~\ref{claim:M(Ki)_upper_bound}, we get that for distinct clumps $K, K'\in \mathcal K$ we have $|V(D(K))\cap V(D(K'))|< \kappa^{6(10\max(k(K), k(K'))!} m\leq  \kappa m$.  
Lemma~\ref{lem:clump_tree_embedding} (2) gives that ``for each clump $K\in \mathcal K$, there are at most $Cp^{-3}h$ clumps $K'\in \mathcal K$ with $|V(D(K))\cap V(D(K'))|\geq \gamma m$''.

\begin{claim}\label{claim:clump_ordering}
We can order  $\mathcal K=(K_1', \dots, K_s')$ such that $|V(D(K_i'))\cap (\bigcup_{j>i} V(D(K_j')))|\leq \alpha d/10$ for all $i$.
\end{claim}
\begin{proof}
Choose a maximal ordered subfamily $\mathcal K'=(K_1', \dots, K_r')$ of $\mathcal K$ with the property that $|V(D(K_i'))\cap (\bigcup_{K\in\mathcal K\setminus  \mathcal K'}V(D(K)) \cup \bigcup_{j>i} V(D(K_j')))|\leq \alpha d/10$ for all $i=1, \dots, r$. If $\mathcal K'\setminus \mathcal K=\emptyset$, then we are done, so suppose that $\mathcal K\setminus \mathcal K'\neq \emptyset$. 
Note that if for some $K\in \mathcal K\setminus \mathcal K'$, we have $|V(D(K))\cap \bigcup_{K'\in \mathcal K\setminus (\mathcal K'\cup \{K\})} V(D(K'))|\leq \alpha d/10$, then we could  define $K'_{r+1}:=K$ to get a bigger $\mathcal K'$, contradicting maximality. So we can suppose that this doesn't happen, which tells us that  $|\mathcal K'\setminus \mathcal K|\neq 1$, and that for all  $K\in \mathcal K\setminus \mathcal K'$, we have $|V(D(K))\cap \bigcup_{K'\in \mathcal K\setminus (\mathcal K'\cup \{K\})} V(D(K'))|> \alpha d/10$.

Let $A$ be the set of vertices of $V(G)$ that are in more than one $V(D(K))$ with $K\in \mathcal K\setminus \mathcal K'$, noting that that $|V(D(K))\cap A|=|V(D(K))\cap \bigcup_{K'\in \mathcal K\setminus (\mathcal K'\cup \{K\})} V(D(K'))|>\alpha d/10$ for all $i$. 
Construct an auxiliary bipartite graph $H$ with parts $A$ and $B=\mathcal K\setminus \mathcal K'$ with $vK$ an edge whenever $v\in V(D(K))$. We have $\delta(B)\geq \alpha d/10=\alpha mh/10\gg 4096 \gamma m(L^{2L})^3$ and $\delta(A)\geq 2$. Also, the property ``for distinct clumps $K, K'\in \mathcal K$ we have $|V(D(K))\cap V(D(K'))|<\kappa m$'', tells us $|N_{common}(K, K')|< \kappa m$ for distinct $K, K'\in B$. Similarly, we've established ``for each clump $K\in \mathcal K$, there are at most $Cp^{-3}h$ clumps $K'\in \mathcal K$ with $|V(D(K))\cap V(D(K'))|\geq \gamma m$'' which tells us that for each $K\in B$, there are $\leq Cp^{-3}h\leq \frac{\alpha d/10}{16\kappa m}$ clumps $K'\in B$ with $|N_{common}(K, K')|\geq \gamma m$.
By Lemma~\ref{lem:find_tree_1_subdivision} (with $d'=\alpha d/10=\alpha hm/10$, $t=\gamma m$, $s=\kappa m$, $k=L^{2L}$), $H$ contains a $1$-subdivision $S$ of a $L$-ary depth $L$ tree with subdivided vertices in $A$.  

Label each vertex in $S\cap A$ as $u_{K_iK_j}$,  where $K_i, K_j$ are the neighbours of $u_{K_iK_j}$ in $S$ (noting that each vertex in $S\cap A$ has exactly two neighbours since it is a subdivided vertex).  
Let $S'$ be the tree on $V(S)\cap B$ formed by contracting the vertices in $A$ into edges. Now Lemma~\ref{lem:clump_tree_embedding} (3) applies to $S'$ giving a copy of $T$.
\end{proof}
 From each $K_i'$ delete the edges of $K_i'\setminus B(K_i')$, the vertices of  $V(K_i')\setminus V(B(K_i'))$, and the  vertices of $V(D(K_i'))\cap (\bigcup_{j>i} V(D(K_j')))$  to obtain a subgraph $J_i$. We claim that $\bigcup_{i=1}^s J_i$ satisfies the lemma. Indeed, by construction, the $J_i$s are vertex disjoint (consider some distinct $J_i, J_j$, noting $V(J_i)\subseteq V(B(K_i'))\subseteq V(D(K_i')), V(J_j)\subseteq V(B(K_j'))\subseteq  V(D(K_j'))$. Without loss of generality $j>i$, which gives $V(J_i)\cap V(J_j)\subseteq V(D(K_i'))\cap V(D(K_j'))\subseteq V(D(K_i'))\cap (\bigcup_{j>i} V(D(K_j')))$. But $J_i$ has no vertices of   $V(D(K_i'))\cap (\bigcup_{j>i} V(D(K_j')))$ giving $V(J_i)\cap V(J_j)=\emptyset$). Since all edges of  $\bigcup_{i=1}^s J_i$ are contained in some $J_i$, we get that there are no edges of $\bigcup_{i=1}^s J_i$ going between distinct $J_i, J_j$, and hence each connected component of $\bigcup_{i=1}^s J_i$ is contained in some $J_i$.
Since each $J_i\subseteq B(K_i')$  has a cover of size $\leq Cd$ (namely $M(K_i')$ by Fact~\ref{fact:BK_small_covers}), the connected components  of $\bigcup J_i$ have covers of size $\leq Cd=(2+\varepsilon)d$.

The number of edges absent from $e(G)$ in $\bigcup_{i=1}^s J_i$ is: the edges of $\bigcup K_i'\setminus B(K_i')$ of which there are $\overset{L:\ref{lem:clump_large_intersection_with_B/D}}{\leq} \sum_{i=1}^s4p e(K_i')= 4p e(G_{dense})\leq 4pe(G)\leq 8pdn$. The edges of each $B(K_i')$ touching the sets $V(D(K_i'))\cap (\bigcup_{j>i} V(D(K_j)))$, of which there are $\leq \sum_{i=1}^s|V(B(K_i'))||V(D(K_i'))\cap (\bigcup_{j>i} V(D(K_j')))|\leq 
\sum_{i=1}^s|V(B(K_i'))|\alpha d/10\overset{F:\ref{lem:clump_lower_bound_eK}}{\leq} \sum_{i=1}^s e(K_i')p^{-13}\alpha d/m= \sum_{i=1}^s \alpha p^{-13}he(K_i')= \alpha p^{-13}h e(G_{dense})\leq \alpha p^{-13}h e(G) \leq  \alpha p^{-13}h  2dn$. The  edges of $G_{sparse}$ of which there are $\leq \varepsilon d n/3$.  These add up to $\leq (2\alpha p^{-13}h  +8p+\varepsilon/3) d n\leq \varepsilon dn$ as required.  
\end{proof}

\section{Concluding remarks}
Theorem~\ref{main_theorem} seems to really open up the area of studying $T$-free graphs $G$. Prior to this theorem, it seemed that the difficult cases of open problems about such $G$ were when $G$ is sparse --- whereas when $G$ is dense he had a pretty good idea for how to approach problems. Now it seems like the situation has partially reversed (at least when one is studying bounded degree trees). Theorem~\ref{main_theorem} explains a lot of the behaviour in the sparse case, and the problem shifts to mainly understanding dense ones. We highlight the most interesting open problems.
\subsubsection*{High degree trees}
Without doubt, the biggest shortfall of this paper is that it only works for bounded degree trees, and the most interesting open problem is to understand arbitrary trees. As mentioned in the introduction our main theorem simply isn't true without some sort of degree restriction. Thus the approach definitely needs to change to prove say the full Erd\H{o}s-S\'os Conjecture, and we leave it to readers to think about what happens there. One concrete open problem is to see how much the bound on $\Delta(T)$ in Theorem~\ref{main_theorem} can be increased before things break down. While the condition can't be removed completely, increasing it to something that is linear in $d$ seems plausible. 

\subsubsection*{Variants of Theorem~\ref{main_theorem}}
It is possible to prove variants of our main theorem which give more  information on the graph  $H$. The most simple of these is to sharpen the bound ``$3d$'' on the sizes of the covers. The proof given in this paper actually gives the slightly better bound of  ``$(2+\varepsilon)d$'' rather than ``$3d$'' --- but this is still not optimal. It is not to hard to sharpen it further to a  near-optimal bound of ``$(1+\varepsilon)d$''.
\begin{theorem}\label{main_theorem_optimal}
Let $\Delta, \varepsilon\gg d^{-1}$.
Let $T$ be a tree with $k$ edges and $\Delta(T) \leq \Delta$. For any graph $G$ with $e(G) \geq \varepsilon k |G|$ having no copies of $T$, it is possible to delete $\varepsilon e(G)$ edges to obtain a graph $H$, each of whose connected components has a cover of order $\leq (1 + \varepsilon)k$.
\end{theorem}
This bound can seen to be near-optimal by considering a graph $G$ of vertex disjoint cliques of order $d-1$. This $G$  has no copy of $T$ because its components are too small. But by deleting $\varepsilon e(G)$ edges, at least one of the cliques will have $\geq (1-\varepsilon)\binom{d-1}2$ edges left --- and so has no cover smaller than $(1-2\sqrt{\varepsilon} )d$ (the maximum number of edges that can touch $(1-2\varepsilon )d$ vertices in a subgraph of $K_{d-1}$ is $\binom{d-1}2-\binom{2\sqrt{\varepsilon} d}2< (1-\varepsilon )\binom{d-1}2$).

We include a derivation of Theorem~\ref{main_theorem_optimal} in the companion paper~\cite{supplementary}. The reason we leave it out is because we are not aware of any additional applications that having an optimal bound gives. Also the proof of Theorem~\ref{main_theorem_optimal} doesn't actually involve modifying anything about the existing proof of Theorem~\ref{main_theorem} --- instead one uses that theorem as a starting point, and then sharpens the structure of the graph $H$ that it gives using tools for analysing dense graphs (namely Szemer\'edi's Regularity Lemma). This means that Theorem~\ref{main_theorem_optimal} ends up having a tower-type dependence between $\varepsilon$ and $d$ (which Theorem~\ref{main_theorem} does not have). 

Going further, one can likely get even stronger variants that give more structural information on $H$. However, again it is unclear what sorts of additional applications variants like this may provide.

\subsubsection*{Further bounded-degree applications}
The applications we gave to the Erd\H{o}s-S\'os and Loebl-Koml\'os-S\'os conjectures are really just the tip of the iceberg for what Theorem~\ref{main_theorem} seems useful for. There are many variants of these problems and  Theorem~\ref{main_theorem} feels like it should help deal with the sparse, bounded degree case of all of them. Below we list all related unsolved problems that we know about undirected graphs. See~\cite{stein2020tree} for a more detailed survey.
\begin{conjecture}[Loebl, Koml\'os, S\'os, see \cite{erdHos1995discrepancy}] \label{conj:LKS_conclusion}
Let $G$ be an $n$-vertex graph with $\geq n/2$ vertices of degree $\geq d$. Then $G$ contains a copy of every $d$-edge tree $T$. 
\end{conjecture}

\begin{conjecture}[Hladky,  Koml{\'o}s, Piguet, Simonovits,  Stein,  Szemer{\'e}di,  \cite{hladky2017approximate1}]
Let $G$ be an $n$-vertex graph with $\geq \frac n2-\lfloor \frac n{d+1}\rfloor- (n\mod{d+1})$ vertices of degree $\geq d$. Then $G$ contains a copy of every $d$-edge tree $T$. 
\end{conjecture}

\begin{conjecture}[Klimo\v{s}ov\'a, Piguet,  Rohzo\v{n}, \cite{klimovsova2020version}] 
Let $G$ be an $n$-vertex graph with $\geq rn$ vertices of degree $\geq d$. Then $G$ contains a copy of every $r$-skew $d$-edge tree $T$. 
\end{conjecture}

\begin{conjecture}[Havet,  Reed,  Stein,  Wood, \cite{havet2020variant}]\label{conj23}
Let $\delta(G)\geq\lfloor2d/3\rfloor$ and $\Delta(G)\geq d$. Then $G$ has a copy of every $d$-edge tree.
\end{conjecture}

\begin{conjecture}[Besomi, Pavez-Sign\'e,  Stein \cite{besomi2019degree}]
Let $\delta(G)\geq d/2$ and $\Delta(G)\geq 2d$. Then $G$ has a copy of every $d$-edge tree.
\end{conjecture}

\begin{conjecture}[Besomi, Pavez-Sign\'e,  Stein \cite{besomi2019degree}]\label{conjgeneraldegrees}
Let $\alpha \in[0,1/3)$,   $\delta(G)\geq(1+\alpha)d/2$ and $\Delta(G)\geq 2(1-\alpha)d$. Then $G$ has a copy of every $d$-edge tree.
\end{conjecture}

\begin{conjecture}[Besomi, Pavez-Sign\'e,  Stein \cite{besomi2019degree}]\label{conjdegrees}

Let $\delta(G)\geq d/2$ and $\Delta(G)\geq 2(1-1/\Delta)k$. Then $G$ has a copy of every $d$-edge tree with $\Delta(T)\leq \Delta$.
\end{conjecture}

\begin{conjecture}[Klimo{\v{s}}ov{\'a},  Piguet,  Rozho{\v{n}}, see  \cite{rozhon2019local}]
Let $\delta(G)\geq d/2$ and suppose that $G$ has at least $|G|/2\sqrt d$ vertices of degree $d$. Then $G$ has a copy of every $d$-edge tree.
\end{conjecture}

\begin{question}[Alon, Shickhelman,  \cite{alon2016many}]
Let $T$ be a $d$-edge tree and $G$ an $n$-vertex graph with no copies of $T$. If $G$ contains more than $\lfloor \frac{n}{d-1}\rfloor\binom{d-1}t+\binom{n-(d-1)\lfloor\frac{n}{d-1}\rfloor}{t}$ copies of $K_t$, then must $G$ contain a copy of $T$?
\end{question}

\begin{problem}[Cioab{\u{a}}, Desai,   Tait,  \cite{cioabua2022spectral}]
For every tree $T$ and $n\in \mathbb{N}$, determine how large the spectral radius of an $n$-vertex, $T$-free graph can be.
\end{problem}
In a future paper with Versteegen and Williams~\cite{vw}, we will apply Theorem~\ref{main_theorem} in order to make progress on some of these. In particular we will asymptotically prove Conjectures~\ref{conj23} --~\ref{conjdegrees} and exactly prove Conjecture~\ref{conj23}, Conjecture~\ref{conjgeneraldegrees} for $\alpha>0$, and Conjecture~\ref{conj:LKS_conclusion} for trees with $\ge\Omega(d)$ leaves (all of these results will be for large, bounded degree trees).

\subsubsection*{Hypergraphs and directed graphs}
There are generalizations of the Erd\H{o}s-S\'os Conjecture to both hypergraphs and directed graphs.

\begin{conjecture}[Kalai, \cite{frankl1987exact}]
Let $\mathcal H$ be an $r$-uniform hypergraph with $n$ vertices and $e(\mathcal H)\geq \frac{k-1}r\binom{n}{r-1}$. Then $\mathcal H$ contains a copy of every tight $r$-uniform hypertree with $k$ edges.
\end{conjecture}
\begin{conjecture}[Addario-Berry, Havet, Linhares Sales, Thomass\'e and Reed, \cite{addario2013oriented}]
Let $D$ be a digraph with $e(D)\geq (k-1)|V(D)|$. Then $D$ has a copy of every antidirected $k$-edge-tree.
\end{conjecture}
Obviously our main theorem doesn't give anything here because it is about undirected $2$-uniform graphs. However, it seems likely that some sort of analogues should hold in the settings of both hypergraphs and directed graphs.
\subsubsection*{Applications in Ramsey theory}
It is well known that  Erd\H{o}s-S\'os Conjecture and its relatives have connections to estimating Ramsey numbers of trees. For example Erd\H{o}s and Graham observed that the Erd\H{o}s-S\'os Conjecture implies an affirmative answer to the following.
\begin{problem}[Erd\H{o}s, Graham  \cite{erdHosgraham}]
Is the $k$-colour Ramsey number of a $k$-edge tree $T$ equal to $rk+O(1)$.
\end{problem}
Using Theorem~\ref{thm:Erdos-Sos-bounded-degree} we get a proof of this being true for bounded degree trees. There are more connections and questions with Ramsey theory, which are all well-documented in~\cite{stein2020tree}.

\subsubsection*{Acknowledgement}
The author would like to thank Kyriakos Katsamakis, Shoham Letzter, Benny Sudakov, Leo Versteegen, and Ella Williams for discussion about this project.

\bibliographystyle{abbrv} 
\bibliography{trees}  

\section{Appendix: tree embedding}

Here we prove Lemma~\ref{lem:main_tree_embedding_lemma}.

The first thing we need is the standard fact that every can be split into  two subtrees with some control over their size. 
In the proof of the next lemma we use notation for rooted trees. We think of a rooted tree as being a digraph with all edges directed away from the root. For a vertex $v$, its out-neighbours are called the children of $v$. If there is a directed path from $u$ to $v$ we say that $v$ is a descendent of $v$. We say it is a strict descendent if additionally $u\neq v$.
\begin{lemma} \label{lem:tree_splitting} 
Let $\Delta\geq 1$ and $0\leq \lambda<1$.
Let $T$ be a rooted $n$-vertex tree with $\Delta(T)\leq \Delta$. There is a subtree $Q$ containing the root with $|Q|\in[(1-2 \lambda )n, (1-\lambda)n+2\Delta]$, such that $T\setminus Q$ has $\leq \Delta$ components, each of which has order $\leq \lambda n$.
\end{lemma} 
\begin{proof}
Let $r$ be the root of $T$. 
Note that if $(1-\lambda)n+2\Delta\geq n$, we can choose $Q=T$, so suppose that   $(1-\lambda)n+2\Delta< n$.
Let $N^i=\{v: d(r,v)=i\}$. Let $t$ be the smallest number such that the components of $T$ descending from $N^t$ have order $\leq \lambda n$ ($t$ exists because for $i>n$, there are no components descending from $N^i$). Note $t\geq 1$, since the component descending from $N^{0}=\{r\}$, is $T$ which has order $n> \lambda n$. By minimality, there is some component $R_y$ descending from $y\in N^{t-1}$ such that $|R_y|> \lambda n$. 
Let $Q_1, \dots, Q_{t}$ be the subtrees descending from the children of $y$. Let  $Q_0$ be  $T$ with all strict descendants of $N(y)$ deleted, noting that this contains the root (since $r$ is not a strict descendent of anything),  is a subtree (for any $v\in Q_0$ the path from the root to $v$ can't contain any strict descendants of $N(y)$ since otherwise $v$ would also be such a descendent), has $E(Q_0)=T\setminus E(Q_1\cup \dots \cup Q_t)$ (since for any edge $uv\in Q_i$, $u$ is a descendent of $N(y)$ and so $v$ is a strict descendent of $N(y)$, giving $v\not\in Q_0$. Meanwhile for any edge $uv\not\in Q_i$, we have that $u$ is not a descendent of $N(y)$ and so $v$ is not a strict descendent of $N(y)$), and has  $V(Q_0)=(T\setminus R_y)\cup \{y\}\cup N(y)$ (strict descendants of $N(y)$ are exactly the descendants of $y$ outside $y\cup N(y)$).
Using $E(T \setminus  Q_0)= E(Q_1\cup \dots \cup Q_t)$, we have that the components of $T \setminus  Q_0$ are exactly $Q_1, \dots, Q_t$ (these are subtrees by definition, and are vertex disjoint since they descend from distinct vertices in $N(y)$).

 Since $Q_1, \dots, Q_{t}$ descend from children of $y$ (which are in $N^t$), they have $|Q_i|\leq \lambda n$. Since $V(Q_0)=(T\setminus R_y)\cup \{y\}\cup N(y)$, it has $|Q_0|\leq n-|R_y|+\Delta+1\leq (1-\lambda)n+2\Delta$.  Now let $Q= \bigcup_{i=0}^{s} Q_i$, where $s$ is chosen maximal such that $|Q|\leq (1-\lambda)n+2\Delta$ ($s$ exists and is $< t$ because $\bigcup_{i=0}^{t}Q_i= n>(1-\lambda)n+2\Delta$). Note that this ensures that $|Q|\geq (1-2\lambda)n$, as otherwise $|Q\cup Q_{s+1}|\leq |Q|+|Q_{s+1}|< (1-2\lambda)n+\lambda n=(1-\lambda)n$ would contradict maximality of $s$. Now $Q$ satisfies the lemma --- indeed the components of $T\setminus E(Q)$ are some subset of $\{Q_1, \dots, Q_t\}$,  so there are $t\leq \Delta$ of these, and they all have order $\leq \lambda n$.
 Also $Q$ is connected (and so a subtree) since the roots of $Q_1, \dots, Q_s$ are contained in $N(y)\subseteq V(Q_0)$.
\end{proof}

\Hide{
\begin{lemma} \label{lem:tree_splitting_more_accurate}
Let $1\geq \Delta^{-1}, \varepsilon\gg n^{-1}$ and $\alpha\in [0,1]$.
Let $T$ be a rooted $n$-vertex tree with $\Delta(T)\leq \Delta$. There is a subtree $Q$ containing the root with $|Q|\in[\alpha n, (\alpha+\varepsilon)n]$ such that $T\setminus Q$ has $\leq \Delta \log_2 (2\varepsilon^{-1})$ components.
\end{lemma} 
\begin{proof}
Starting with $T_0=T$, construct a sequence of trees $T_1, T_2, \dots$ as follows: Apply Lemma~\ref{lem:tree_splitting} to $T_i$ with $\lambda=\frac{|T_i|-\alpha n}{2|T_i|}$ to get $T_{i+1}$. Note that this inductively ensures that $|T_i|\geq \alpha n$ always, and also that $|T_{i+1}|\leq (1-\lambda)|T_i|+2\Delta=\frac{|T_i|+\alpha n}{2}+2\Delta$. The latter rearranges to $|T_{i+1}|-\alpha n\leq \frac{|T_i|-\alpha n}{2}+2\Delta$. Setting $i=\log_2 (2\varepsilon^{-1})$ gives $|T_i|-\alpha n\leq  n/2^i+2i\Delta= \varepsilon n/2+4\varepsilon^{-1}\Delta\leq \varepsilon n$.
\end{proof} 
}
Our basic strategy for tree embedding is ``greedily embed $T$ into a graph $G$ using Fact~\ref{fact:greedy_tree_embedding}. If we ever run out of space, move to some other part of $G$ using cut-density''. The ``move to some other part of $G$'' is done via the following.
\begin{lemma}\label{lem:d_ary_tree_with_prescribed_leaves}
Let $1\geq \Delta^{-1}, q\gg L^{-1}\gg n^{-1}$.
Let $G$ be an order $n$, $q$-cut-dense graph, $v\in V(G)$, $U\subseteq V(G)$ with $|U|\geq L$. There is a perfect $\Delta$-ary tree of height $\leq 19q^{-4}$ rooted at $v$ with all leaves in $U$. 
\end{lemma}
\begin{proof}
Note Fact~\ref{fact:cut_dense_min_degree} we have $\delta(G)\geq qn/2$.
Define $N^t$ recursively by setting $N^1=N(v)$ and $N^{t+1}=N(v)\cup \{u: |N(u)\cap  N^{t}|\geq q^4 n/16\}$. Note that $N^{i}\subseteq N^{i+1}$ for $i=1,2,\dots, $ (this holds by induction on $i$. The initial case is $N^1=N(v)\subseteq N^2$. For the inductive step, use that ``if $N^{i-1}\subseteq N^i$, then $\{u: |N(u)\cap  N^{i-1}|\geq q^4 n/16\}\subseteq \{u: |N(u)\cap  N^{i}|\geq q^4 n/16\}$'').
Note  $|N^i|\geq \delta(G)\geq qn/2$ for all $i$.
Note that for every $t$, we either have $|N^t|\geq n-q^2 n/4$ or $|N^t|\geq |N^{t-1}|+q^4n/16$. Indeed, if the former doesn't happen, then the definition of $q$-cut-dense gives $e(N^{t-1}, G\setminus N^{t-1})\geq q|N^{t-1}|(n-|N^{t-1}|)\geq q(qn/2)(q^2n/4)=q^4n^2/8$. Fact~\ref{fact:subgraph_min_degree_e/2} gives a subgraph $H$ of $G[N^{t-1}, G\setminus N^{t-1}]$ with $\delta(H)\geq q^4n/16$. This gives $q^4n/16$ vertices in  $G\setminus N^{t-1}$, each with $q^4n/16$ neighbours in $N^{t-1}$. Each of these must be  in  $N^{t}\setminus N^{t-1}$, giving  $|N^t|\geq |N^{t-1}|+q^4n/16$.
Fix $t_0$ to be the smallest number for which $|N^{t_0}|\geq n-q^2 n/4$, noting that $t_0\leq 18q^{-4}$ (as otherwise $|N^{t_0-1}|\geq|N^{t_0-2}|+q^4n/16\geq |N^{t_0-3}|+2\cdot q^4n/16\geq \dots \geq |N^{1}| +(t_0-2)q^4n/16> (18q^{-4}-2)(q^4n/16)\geq (16q^{-4})(q^4n/16)= n$, which is impossible).

 Setting $A_{t_0+1}:=U$, we define $s$ and $A_{t_0}\subseteq N^{t_0}, A_{t_0-1}\subset N^{t_0-1}, A_{t_0-2}\subseteq N^{t_0-2}, \dots, A_{s}\subseteq N^s$ recursively. To define $A_k$, having already defined $A_{k+1}$: If $|A_{k+1}\cap N(v)|\geq |A_{k+1}|/2$, stop with $s=k+1$.
Otherwise, let $A_{k}=\{u\in N^{k}: |N(u)\cap A_{k+1}|\geq  q^4|A_{k+1}|/128\}$. 

Note that we certainly stop with some $s\geq 1$ (since $A_1\subseteq N^1=N(v)$ implies $|A_{1}\cap N(v)|=|A_1|\geq |A_{1}|/2$, so we would always stop at $k=1$ if we don't stop earlier).
Note  that for $k\in[s, t_0-1]$, we have  $e(A_{k+1}, N^k)\geq \frac12\sum_{a\in A_{k+1}\setminus N(v)}|N(a)\cap N^k|\geq \frac12|A_{k+1}\setminus N(v)|q^4 n/16\geq |A_{k+1}|q^4 n/64$ (the first inequality comes from the handshaking lemma, the second inequality uses $A_{k+1}\subseteq N^{k+1}$ and the definition of $N^{k+1}$, while the third inequality uses $k\geq s$), while for $k=t_0$, we have $e(A_{t_0+1}, N^{t_0})\geq |A_{t_0+1}|(\delta(G)-|A_{t_0+1}|-|V(G)\setminus N^{t_0}|)\geq |A_{t_0+1}|(qn/2-L-q^2 n/4)\geq |A_{t_0+1}|qn/8\geq |A_{t_0+1}|q^4 n/64$. 
Note also that for $k\in[s,t_0]$ we have  $e(A_{k+1}, N^k)\leq e(A_{k+1}, A_k)+e(A_{k+1}, N^k\setminus A_{k})\leq
|A_k||A_{k+1}|+|N^k\setminus A_k|q^4|A_{k+1}|/128\leq  |A_{k}||A_{k+1}|+nq^4|A_{k+1}|/128$ (using the definition of $A_k$ in the second-last inequality).
 Combining and rearranging gives $|A_k|\geq \frac{|A_{k+1}|q^4 n/64-nq^4|A_{k+1}|/128}{|A_{k+1}|}=  nq^4/128\geq L$ for $k\in [s, t_0]$, and so  $|A_k|\geq L$ for $k\in [s, t_0+1]$.

By definition of $s$ we have $|A_s\cap N(v)|\geq |A_s|/2\geq  L/2\geq \Delta+1$.  Thus we can define a star $T_{s}$ centered at $v$ with $\Delta$ leaves that are all in $A_{s}\cap N(v)$. We define  trees $T_{s+1}, T_{s+2},\dots, T_{t_0+1}$  such that $T_i$ is a perfect height $(i-s+1)$-ary, $\Delta$-ary (and so order $\leq \Delta^{t_0+2}$) tree rooted at $v$ with all leaves in $A_i$ recursively as follows: having defined $T_i$, we have that each leaf of $T_i$ has $\geq q^4|A_{i+1}|/128\geq q^4L/256\gg\Delta^{t_0+2}\times\Delta$ neighbours in $A_{i+1}$ (by definition of $A_{i}$ and since leaves of $T_i$ are in $A_i$), so we can greedily choose $\Delta$ such neighbours for each vertex one by one. Since $A_{t_0+1}=U$ and $t_0+1\leq 18q^{-4}+1\leq 19q^{-4}$,  $T_{t_0+1}$ satisfies the lemma.
\end{proof}

Given a subtree $T'\subseteq T$, a $T$-external vertex of $T'$ is any vertex $v\in V(T')$ which touches at least one edge of $T\setminus T'$. Use $external(T,T')$ to denote the set of $T$-external vertices of $T'$. 
We'll use the fact that there is a natural bijection between the $T$-external vertex of $T'$  and the connected components of $T\setminus T'$.
\begin{fact}\label{fact:T_external_vertices}
For trees with $T'\subseteq T$, every component of $T\setminus T'$ contains precisely one $T$-external vertex of $T'$, and every $T$-external vertex of $T'$ is contained in precisely one component of $T\setminus T'$.
\end{fact}
\begin{proof}
Let $T_1, \dots, T_a$ be the components of $T\setminus T'$, and $y_1, \dots, y_b$ the $T$-external vertices of $T'$. By definition, for every $y_i$, there is an edge $y_iz_i\in E(T)\setminus E(T')$, and this edge must be contained in some $T_j$. We can't have $y_i$ in more than one $T_j$  since connected components are vertex-disjoint. 

Consider some $T_j$. Since  $T\setminus T'$ has no isolated vertices (by definition), $T_j$ contains at least one edge $uv$. This shows that any vertex  $y\in V(T')\cap V(T_j)$ is a $T$-external vertex of $T'$ (it must be contained in an edge  $xy\in E(T_j)\subseteq E(T)\setminus E(T')$ since $T_j$ is connected). If there was more than one vertex  $x,y\in V(T')\cap V(T_j)$, then there'd be a circuit in $T$ (formed by joining the paths between $x,y$ in $T'$ and $T_j$). Also, it's impossible that $V(T')\cap V(T_j)=\emptyset$. Indeed otherwise, by connectedness, there'd be some edge $uv\in E(T)$ with $u\in V(T_j), v\not\in V(T_j)$. Since $V(T')\cap V(T_j)=\emptyset$, we must have $uv\not\in E(T'), E(T_j)$ --- which contradicts $T_j$ being a connected component of $T\setminus T'$. Thus $|V(T')\cap V(T_j)|=1$, and the vertex inside this set is a $T$-external vertex of $T'$ as required. 
\end{proof}

\begin{fact}\label{fact:external_components_intersections}
Let $T'\subseteq T$ be trees, $y_1, \dots, y_t$ the $T$-external vertices of $T'$ and $T_{y_1}, \dots, T_{y_t}$ the corresponding components of $T\setminus T'$.
Then $V(T_{y_i})\cap V(T_{y_j})=\emptyset$ for $i\neq j$ and $V(T_{y_i})\cap V(T')=\{y_i\}$ for all $i$.
\end{fact} 
\begin{proof}
For $V(T_{y_i})\cap V(T_{y_j})=\emptyset$ for $i\neq j$, just note that $T_{y_i}, T_{y_j}$ are connected components of the graph $T\setminus T'$ (and hence disjoint, because connected components are always vertex-disjoint). We have $y_i\in T'$ by definition of ``$y_i$ $T$-external vertex in $T'$'' and $y_i\in T_{y_i}$ by definition of $T_{y_i}$. Suppose that there was some other vertex $x\in V(T')\cap V(T_{y_i})$. There are paths $P\subseteq T', Q\subseteq T_{y_i}$ joining $x$ and $y_i$ (since $T', T_{y_i}$ are trees). These paths are edge-disjoint (since $T', T_{y_i}$ are edge-disjoint). Thus $P\cup Q$ is a circuit in $T$, which is a contradiction. 
\end{proof} 
 
The following will be used to combine several embeddings of subtrees into an embedding of a single tree. 
  \begin{fact}\label{fact:combine_tree_embeddings}
Let $Q_0\subseteq T$ be trees,  let $y_1, \dots, y_t$ be the $T$-external vertices of $Q_0$, and $Q_{1}, \dots, Q_{t}$ the components of $T\setminus Q_0$ with $y_i\in Q_i$. Suppose that we have embeddings $f_0, f_1, \dots, f_t$ with $f_i:Q_i\to S$ with the properties that $f_0(y_i)=f_i(y_i)$ for $i=1, \dots, t$ and  $im_{f_i}\cap im_{f_j}\subseteq \{f_0(y_i)\}$ for  $i>j$.

Then we can define an embedding $f:Q \to T$ by having $f$ agree with $f_i$ on each $Q_i$.
 \end{fact} 
 \begin{proof}
 First note that this is well-defined: indeed for $i>j$ we have $V(Q_i)\cap V(Q_j)=\{y_i\}$ if $j=0$ and   $V(Q_i)\cap V(Q_j)=\emptyset$ if $j\geq 1$ (all by Fact~\ref{fact:external_components_intersections}). Since $f_i$ and $f_0$ agree on $y_i$, $f$ is well defined on $T=\bigcup_{i=1}^t Q_i$. 
 To see that $f$ is a homomorphism, consider some edge $xy\in T$. Since $E(T)=\bigcup_{i=1}^t E(Q_i)$, we have $xy\in Q_i$ for some $i$. By definition $f(x)f(y)=f_i(x)f_i(y)$ which is an edge of $T$ since $f_i:Q_i\to S$ was a homomorphism.
 
To see that $f$ is injective, consider distinct $x,y\in T$. If $x,y\in Q_i$  for some $i$, then $f(x)=f_i(x)\neq f_i(y)=f(y)$ because $Q_i$ was injective. So suppose $x\in Q_i, y\in Q_j$ with $i> j$. 
We have $im_{f_i}\cap im_{f_j}\subseteq \{f_0(y_i)\}$ so we could only have $f(x)=f_i(x)=f_j(y)=f(y)$ if  $f(x)=f(y)=f_0(y_i)$.  
If $j\geq 1$, then we have $im_{f_j}\cap im_{f_0}\subseteq \{f_0(y_j)\}$. Since $f_0$ is injective we have $f_0(y_i)\neq f_0(y_j)$ and so it's impossible that $f(y)=f_j(y)=f_0(y_i)$.
So suppose $j=0$. Since $f_i$ is injective and $f_i(y_i)=f_0(y_i)=f_i(x)$, we have $x=y_i$. Since $f_0$ is injective, $f(y)=f_0(y)\neq f_0(y_i)=f_i(y_i)=f_i(x)=f(x)$.
 \end{proof}
 
 We also need the following.
 \begin{fact} \label{fact:maximal_subtree_extermal_vertices}
Let $T,Q$ be trees rooted at $r_T, r_Q$ respectively with $\Delta(T)\leq \Delta$, and $Q$ a $\Delta$-ary tree. Let $T'\subseteq T$ be a maximal subtree containing $r_T$ for which there is an embedding $f:T'\to Q$ with $f(r_T)=r_Q$. Then $f$ embeds $T$-external vertices of $T'$ into leaves of $Q$. 
 \end{fact}
 \begin{proof}
 Suppose for contradiction that we have  some $T$-external vertex of $u\in T'$ with $f(u)$ not a leaf of $Q$. Since $u$ is $T$-external, there is an edge $uv$ with $v\in V(T)\setminus V(T')$. Since $f(u)$ is not a leaf and $Q$ is $\Delta$-ary, we have $d_Q(f(u))=\Delta$. Since $f(T')$ is a tree with maximum degree $\leq \Delta$, there is some edge $f(u)y\in E(Q)$ through $f(u)$, which is not in $E(f(T'))$. Since $f(T')$ is a subtree of $Q$, and $f(u)\in V(f(T'))$, this means $y\not\in V(f(T'))$ (since the unique path from $y$ to $f(u)$ in the tree $Q$ is $yf(u)\not\in E(f(T'))$). Thus we can extend $f$ to embed $v$ to $y$, contradicting maximality. 
 \end{proof}
 
 The following lemma is basically the inductive step of our proof of Lemma~\ref{lem:main_tree_embedding_lemma} (i).
\begin{lemma} \label{lem:tree_embedding_into_dense_induction_step}
Let $1\geq p,q, \alpha, \Delta^{-1}\gg n^{-1}$.
Let $G$ be an $n$-vertex, $q$-cut-dense graph that contains  a subgraph $R$ with $\delta(R)\geq pn/4$. 
Let $T$ be an order $\leq pn/4$   tree rooted at $r_T$ with $\Delta(T)\leq \Delta$, and $r\in V(G)$. Then an embedding $f:T\to G$ with $f(r_T)=r$.
\end{lemma} 
\begin{proof}
Lemma~\ref{lem:d_ary_tree_with_prescribed_leaves} gives an order $\leq \Delta^{20q^{-4}}$, $\Delta$-ary tree $T_Q$ rooted at $r$ with leaves in $R$ (for this application, pick some $L$ with $q, \Delta^{-1}\gg L^{-1}\gg n^{-1}$, then note $|R|\geq pn/4\gg L$ so we can use $U=R$).
Let $T'$ be a maximal subtree of $T$  rooted at $r_T$ for which there is an embedding $g:T'\to T_Q$ with $g(r_T)=r$. By maximality, via Fact~\ref{fact:maximal_subtree_extermal_vertices}, all $T$-external vertices $y_i\in T'$ have $g(y_i)$ a leaf of $T_Q$ (and so $g(y_i)\in R$).
Let $y_1, \dots, y_k$, be the $T$-external vertices of $T'$, and $T_{y_1}, \dots, T_{y_k}$ be the components of $T\setminus T'$ that contain them. For $i=1,\dots, k$, use Fact~\ref{fact:greedy_tree_embedding}  to get an embedding $f_{y_i}: T_{y_i}\to (R\setminus (im_g\cup \bigcup_{j<i} im_{f_{y_i}}))\cup \{g(y_i)\}$ with $f_{y_i}(y_i)=g(y_i)$ (for this, we use that the set we delete has size $|im_g\cup \bigcup_{j<i} im_{f_{y_i}}|=e(T')+\sum_{j=1}^{i-1}e(T_{y_j})+1\leq e(T)-e(T_{y_i})+1=|T|-e(T_i)$, and so $R$ still has degree $\geq \delta(R)-|im_g\cup \bigcup_{j<i} im_{f_{y_i}}|\geq pn/4-(pn/4-e(T_{y_i})= e(T_{y_i})$ outside this set). Now set $f$ to equal $g$ on $T'$ and equal $f_{y_i}$ on each $T_{y_i}$, noting that this is an embedding by Fact~\ref{fact:combine_tree_embeddings}.
\end{proof}

The following is Lemma~\ref{lem:main_tree_embedding_lemma} (i). 
\begin{lemma} \label{lem:tree_embedding_into_dense}
Let  $p, q,\Delta^{-1}, \varepsilon\gg  n^{-1}$.
 and let $T$ be a  tree with $\Delta(T)\leq \Delta$, and $G$ a $n$-vertex, $q$-cut-dense graph,  that contains a non-empty $pn$-regular subgraph $R$ with $|R|\geq (2+2\varepsilon)|T|$. Then $G$ has a copy of $T$.
\end{lemma} 
\begin{proof}
Pick $p, q,\Delta^{-1}, \varepsilon\gg \alpha\gg n^{-1}$.
For $1\leq i\leq \Delta, 1\leq j\leq 80p^{-1}$, let $Q^i_j$, be disjoint $\alpha$-random subsets of $V(G)$ (we can choose them disjointly because $80\alpha p^{-1}\Delta\ll 1$), and set $Q=V(G)\setminus \bigcup Q^i_j$.  Note that with high probability, these all satisfy the conclusion of Lemma~\ref{lem:cut_dense_random_sample_plus_more} and have $|Q_i^j|\leq 2\alpha n$ (by Chernoff's bound for the latter). Set $G_k=Q\cup \bigcup_{i\in[\Delta], j\leq k} Q^i_j$. 
We prove the following for $k=0, \dots, \lfloor\frac{80\varepsilon^{-1} p^{-1}}{2+\varepsilon}\rfloor$ by induction on $k$:
\begin{itemize}
\item[$(\ast)$] Let $T$ be a tree with $|T|\leq \frac{kp\varepsilon |R|}{80}$, $\Delta(T)\leq \Delta$. Then $G_k$ contains a copy of $T$.
\end{itemize}
Note that the $k=\lceil\frac{80\varepsilon^{-1} p^{-1}}{2+2\varepsilon}\rceil\leq\frac{80\varepsilon^{-1} p^{-1}}{2+2\varepsilon}+1 \leq \frac{80\varepsilon^{-1} p^{-1}}{2+\varepsilon} -1\leq \lfloor\frac{80\varepsilon^{-1} p^{-1}}{2+\varepsilon}\rfloor$ case of this implies the lemma (here the middle inequality is true because it rearranges to $2\leq \frac{80p^{-1}}{(2+\varepsilon)(2+2\varepsilon)}$). 
The initial case will be ``$k\leq 2$'', which holds because in this case $|T|\leq \frac{2p|R|}{80}\leq np/40\leq pn-2\alpha \cdot p^{-1}\Delta \cdot n\leq \delta(R)-|\bigcup Q^i_j|$, and so Fact~\ref{fact:greedy_tree_embedding} gives an embedding of $T$ into $R\setminus \bigcup Q^i_j\subseteq G_k$. So suppose that $k\geq 3$, and the lemma holds for smaller $k$.
Apply Lemma~\ref{lem:tree_splitting} to $T$ with $\lambda=2/k$ in order to get a  subtree $T'\subseteq T$  with $|T'|\leq  (1-\lambda)|T|+2\Delta=\frac{k-2}{k}|T|+2\Delta\leq \frac{k-2}{k}\frac{kp\varepsilon|R|}{80}+2\Delta=\frac{(k-1)p\varepsilon|R|}{80}-\frac{p\varepsilon|R|}{80}+2\Delta\leq \frac{(k-1)p\varepsilon|R|}{80}$ (using $2\Delta\ll \frac{p^2\varepsilon n}{80}\leq \frac{p\varepsilon|R|}{80}$ for the last inequality), such that the components of $T\setminus T'$ have order $\leq \lambda |T|$. 
By induction there is an embedding $f_0:T'\to G_{k-1}$.

List  the $T$-external vertices of $T'$ as $\{y_1, \dots, y_{t}\}$ with $t\leq \Delta$. For each $y_i$ let $T_{y_i}$ be the component of $T\setminus T'$, containing $y_i$, noting that $|T_{y_i}|\leq \lambda |T|=\frac2k|T|\leq \frac2k \frac{kp\varepsilon|R|}{80}\leq \frac{p\varepsilon n}{40}$. For $i=1, \dots, t$, we will construct   embeddings $f_i:T_{y_i}\to H_i:=(G_{k-1}\setminus \bigcup_{j=0}^{i-1}im_{f_j})\cup Q^i_k\cup\{f_0(y_i)\}$ that satisfy $f_i(y_i)=f_0(y_i)$.
To construct $f_i$ (having already constructed $f_1, \dots, f_{i-1}$): first notice that $H_i$ is $\alpha^{q^{-4}}q^5$-cut-dense, which comes from Lemma~\ref{lem:cut_dense_random_sample_plus_more}. Next, notice that $R':=R\cap H_i$ has $e(R')\geq e(R)-pn|\bigcup Q_{s}^t\cup  im_{f_j}|\geq pn|R|/2-pn|T|-pn(80 p^{-1}\Delta)(2\alpha n)\geq pn|R|/2-pn\frac{80\varepsilon^{-1} p^{-1}}{2+\varepsilon} \frac{p\varepsilon|R|}{80}-160p\Delta\alpha n =pn\varepsilon|R|\frac1{4+2\varepsilon} -160p\Delta\alpha n\geq pn\varepsilon|R|/8\geq pn\varepsilon|R'|/8$, which shows that $R'$ has a subgraph with minimum degree $\geq p\varepsilon n/16$ via Fact~\ref{fact:subgraph_min_degree_e/2}. Thus Lemma~\ref{lem:tree_embedding_into_dense_induction_step} applies to $H_i$ (with $p'=p/16$), which gives an embedding $f_i: T_{y_i}\to H_i$ with $f_i(y_i)=f_0(y_i)$ as required. 
Now Fact~\ref{fact:combine_tree_embeddings}, applied to the embeddings $f_0, f_1, \dots, f_t$ gives an embedding $f:T\to G$ (for the intersection condition, note that for $a<b$, $im_{f_a}\cap im_{f_b}\subseteq im_{f_a}\cap V(H_b)= im_{f_a}\cap ((V(G_{k-1})\setminus \bigcup_{j=0}^{b-1}im_{f_j})\cup Q^b_k\cup\{f(y_b)\})\subseteq im_{f_a}\cap (Q^b_k\cup\{f(y_b)\})\subseteq \{f(y_b)\}$, where the last containment uses that only $im_{f_b}$ can intersect $Q^b_k$). 
\end{proof}

 The following is basically the inductive step of our proof of Lemma~\ref{lem:main_tree_embedding_lemma} (ii).
 \begin{lemma} \label{lem:tree_embedding_into_tree_induction_step}
 Let $1\geq \Delta^{-1}, q\gg L^{-1}\gg n^{-1}$.
Let $G$ be of order $\geq n$, $T$ a tree with $\Delta(T)\leq \Delta$ rooted at $r_T$, and $T'$ a subtree containing $r_T$, $|T'|\leq qn/8$, such that $T\setminus T'$ has $\leq \Delta$ components. Let $r\in V(G)$, $U\subseteq V(G)$ with $|U|\geq L$. 
Let $g:T' \to G$ be an embedding such that $G\setminus g(T'\setminus external(T,T'))$ is $q/2$-cut-dense.
There is a tree $T''$ with $T'\subseteq T''\subseteq T$ and an embedding $f:T''\to G$ extending $g$, mapping all $T$-external vertices of $T''$ into $U$. 
 \end{lemma}
\begin{proof}
Let $Y$ be the set of $T$-external vertices of $T'$ noting $|Y|\leq \Delta$. 
  We can use Lemma~\ref{lem:d_ary_tree_with_prescribed_leaves} to find  a family of  vertex-disjoint order $\leq \Delta^{19(q/2)^{-4}}$, $\Delta$-ary trees $\{Q_y: y\in Y\}$ with  $V(Q_y)\subseteq (G\setminus g(T'))\cup g(y)$ such that $Q_y$ is rooted on $g(y)$ and has all its leaves in $U$ (first enumerate $Y=\{y_1, \dots, y_t\}$. To construct $Q_{y_i}$ apply Lemma~\ref{lem:d_ary_tree_with_prescribed_leaves} with $G'=(G\setminus (g(T')\cup \bigcup Q_{y_i}))\cup g(y_i)$, $U'=U\setminus \bigcup Q_{y_i}$, $v=y_i$ which ensures the disjointness we want. To check the conditions of Lemma~\ref{lem:d_ary_tree_with_prescribed_leaves}: $|external(T,T') \cup \bigcup Q_{y_i}|\leq qn/8$, and so  Lemma~\ref{lem:cut_dense_delete_small_set} applied to $G\setminus g(T'\setminus external(T,T'))$ shows that $G'$ is $q/4$-cut-dense. Also $|\bigcup Q_{y_i}|\leq |Y|\Delta^{19(q/2)^{-4}}\leq \Delta^{20(q/2)^{-4}}\leq L/2$ and so $|U'|\geq L/2$). Note that this ensures that $V(Q_y)\cap g(V(T'))=\{g(y)\}$ and $V(Q_y)\cap V(Q_{y'})=\emptyset$ for distinct $y,y'\in Y$. 

For each $y\in Y$, let $T_y$ be the component of $T\setminus T'$ containing $y$. For each $y\in Y$, pick a maximal subtree $T_y'\subseteq T_y$ containing $y$ for which there is an embedding $f_y:T_y'\to Q_y$  that embeds $y$ to $g(y)$. By Fact~\ref{fact:maximal_subtree_extermal_vertices}, $f_y$ maps $T_y$-external vertices of $T_y'$ to leaves of $Q_y$ (which are contained in $U$).
Let $T'':= T'\cup \bigcup T_y'$, noting that $V(T'')\cap V(T_y)= V(T_y')$ for all $y$ (since Fact~\ref{fact:external_components_intersections} gives $V(T_y)\cap V(T_z')\subseteq V(T_y)\cap V(T_z)=\emptyset$ for distinct $y,z\in Y$, and $V(T_y)\cap V(T')=\{y\}\subseteq V(T_y')$, and that $T$-external vertices of $T''$ are all $T_y$-external vertices of $T_y'$ for some $y$ (consider a  $T$-external vertex  $u\in T''$, and let $uv$ be an edge of $T$ with $v\not\in T''$.  
 Since $uv\not\in T'\subseteq T''$ and $T=T'\cup \bigcup T_y$, we have $uv\in T_y$ for some $y$. Since $V(T'')\cap V(T_y)= V(T_y')$ for all $y$, this gives $u\in T_y'$. Also $uv\not\in T''$ implies $uv\not\in T_{y'}\subseteq T''$. Hence $u$ is a $T_y$-external vertex of $T_y'$).
 
 Define $f:T''\to G$ to agree with $g$ on $T'$ and agree with $f_y$ on each $T_y'$ to get an embedding via Fact~\ref{fact:combine_tree_embeddings} (for the intersection condition, note that $f_y(T_y')\cap f_z(T_z')\subseteq Q_y\cap Q_z=\emptyset$ and $g(T')\cap f_y(T_y')\subseteq g(T')\cap Q_y= \{g(y)\}=\{f_y(y)\}$ for distinct $y,z\in Y$). It satisfies the lemma because $T$-external vertices of $T''$ are $T_y$-external vertices of some $T_y'$ which are embedded to $U$ by $f_y$ (and hence by $f$).
\end{proof} 
 
 The following is Lemma~\ref{lem:main_tree_embedding_lemma} (ii).
\begin{lemma} \label{lem:tree_embedding_into_tree}
Let  $q,\Delta^{-1}\gg L^{-1}\gg n$ and $k\in \mathbb N$.
 and let $T$ be a order $\leq kqn/16$ tree with root $r_T$, and $\Delta(T)\leq \Delta$. Let $G$ be a graph  with $E(G)=G_1\cup \dots \cup G_t$ with each $G_i$ is a $\geq n$-vertex $q$-cut-dense, and there is an auxiliary complete $L$-ary, height $k$ tree $S$ with $V(S)=\{G_1, \dots, G_t\}$ and a set of distinct vertices $\{u_{G_iG_j}: G_iG_j\in E(S)\}\subseteq V(G)$ such that  $V(G_i)\cap V(G_j)=\{u_{G_iG_j}\}$ for all edges $G_iG_j\in E(S)$ and $V(G_i)\cap V(G_j)=\emptyset$ for non-edges $G_iG_j\not\in E(S)$. 
Let $G_{1}$ be the root of $S$, and $r\in V(G_{1})$.
 
  There is an embedding $f:T\to G$ with $f(r_T)=r$. 
\end{lemma} 
 \begin{proof}
 We fix $q,\Delta, h, L, n$, and prove the lemma by induction on $k$. In the initial case $k=1$, note that Fact~\ref{fact:cut_dense_min_degree} gives $\delta(G_{1})\geq qn/2$, and so we can find a copy of $T$ rooted at $r$ using Fact~\ref{fact:greedy_tree_embedding}. Suppose that $k\geq 2$, and that the lemma holds for smaller $k$.
 
Set $U:=\{u_{G_{1}G_i}: G_i\in N_S(G_1)\}$ to get a set of order $L$. 
Apply Lemma~\ref{lem:tree_splitting} with $\lambda=1-1/k$ in order to find a subtree $T'$ of $T$  of  order $|T'|\leq (1-\lambda)|T|+2\Delta=\frac{|T|}{k}+2\Delta\leq \frac{kqn/16}{k}+2\Delta\leq qn/8$, and so that there are $\leq \Delta$ components of $T\setminus T'$, each of order $\leq \lambda |T|= \frac{k-1}{k}|T|\leq \frac{k-1}{k}kqn/16=(k-1)qn/16$. 
Note that Fact~\ref{fact:cut_dense_min_degree} gives $\delta(G_1\setminus U)\geq qn/2$, and so we can find an embedding $\hat g:T'\to G_{1}\setminus U$ with $g(r_T)=r$ using Fact~\ref{fact:greedy_tree_embedding}. 
Note that $G\setminus g(T'\setminus external(T,T'))$ is still $q/2$-cut-dense (by Lemma~\ref{lem:cut_dense_delete_small_set}). 
Apply Lemma~\ref{lem:tree_embedding_into_tree_induction_step} to get a $T''$ with $T'\subseteq T''\subseteq T$ and an embedding $g:T''\to G_{1}$ extending $\hat g$, which embeds $T$-external vertices of $T''$ into $U$

For each $T$-external $z$ in $T''$, define the following: let $T_z$ be the component of $T\setminus T''$ containing $z$, noting that since $T_z$ is contained in some component of $T\setminus T'$, we have $|T_z|\leq (k-1)qn/16$. 
Since $g(z)\in U$, there is some $G_{i(z)}\in N_S(G_1)$ with $g(z)=u_{G_1G_{i(z)}}$. 
Let $S_z$ be the subtree of $S$ descending from $G_{i(z)}$, noting that this is a complete $L$-ary tree of height $k-1$ (since $G_{i(z)}$ is a neighbour of the root $G_1$ of the  $S$, which is a complete $L$-ary tree of height $k$). 
Let $H_z=\bigcup_{G_i\in S_z} G_i$. Since the only edge from $S_z$ to $G_1$ is $G_1G_{i(z)}$, we get $V(G_1)\cap V(H_z)=\{u_{G_1G_{i(z)}}\}=\{g(z)\}$. 
Use induction to get an embedding $g_z:T_z\to H_z$ with $g_z(z)=g(z)$. 

Finally we apply Fact~\ref{fact:combine_tree_embeddings} to combine $g$ and all $g_z$ into an embedding $f:T\to G$. All conditions have been checked in the previous paragraph, save the last one which we check now: 
Note that for distinct $T$-external $z,w\in T''$, we have $u_{G_1G_{i(z)}}=g(z)\neq g(w)=u_{G_1G_{i(w)}}$ (since $g$ is injective), which gives $G_{i(z)}\neq G_{i(w)}$. Since $G_{i(z)}, G_{i(w)}$ are distinct children of the root, this shows that the trees $S_z, S_w$ descending from them are disjoint with no edges of $S$ between them. This shows that for all $G_i\in S_z, G_j\in S_w$ we have $V(G_i)\cap V(G_j)=\emptyset$, and hence $H_z, H_w$ are vertex-disjoint. Thus $im_{g_z}\subseteq H_z, im_{g_w}\subseteq H_w$ are vertex disjoint too.  
Also, for any $z$ we have that the tree $S_z$ descending  from $G_{i(z)}$  doesn't contain the root $G_1$ (since $G_1$ is a parent of $G_{i(z)}$) and has exactly one edge  going into $G_1$, namely $G_1G_{i(z)}$ (there can't be more than one edge since that would create a cycle in $S$). This means that $V(H_z)\cap V(G_1)=\bigcup_{G_i\in S_z} G_i\cap G_1=\bigcup_{G_i\in S_z, G_iG_1\in E(S)} G_i\cap G_1=\{u_{G_1G_{i(z)}}\}$ and hence $im_g\cap im_{g_z}\subseteq V(H_z)\cap V(G_z)\subseteq \{u_{G_1 G_{i(z)}}\}$. 
 \end{proof}

We end by putting things together to give Lemma~\ref{lem:main_tree_embedding_lemma} as stated in the proof overview.
\begin{proof}[Proof of Lemma~\ref{lem:main_tree_embedding_lemma}]
\begin{enumerate}[(i)]
\item Note that $n=|G|\geq |R|\geq (2+\varepsilon)d$ gives $p, q,\Delta^{-1}, \varepsilon\gg  n^{-1}$ and $|R|\geq (2+\varepsilon)|T|$. Thus Lemma~\ref{lem:tree_embedding_into_dense} applies (with $\varepsilon'=\varepsilon/2$) to give a copy of $T$.
\item Let $n:=16d/qL$, noting $q,\Delta^{-1}\gg L^{-1}\gg n^{-1}$ and $|T|\leq Lqn/16$. Thus Lemma~\ref{lem:tree_embedding_into_tree} applies with $k=L$ to give a copy of $T$.
\end{enumerate}
\end{proof}

 \end{document}